\documentclass[amssymb,10pt]{article}
\usepackage[dvipdfm]{graphicx}
\usepackage{amsthm}
\usepackage[usenames]{color}
\usepackage{xspace,colortbl}
\usepackage{latexsym,url,amscd,amssymb}
\usepackage{amsmath}
\usepackage{graphics}
\usepackage{epstopdf}
\usepackage{hyperref}
\usepackage{subcaption}
\usepackage{multirow}
\usepackage{cases}
\usepackage{caption}
\usepackage{algorithm}
\usepackage{algpseudocode}
\usepackage{listings}
\usepackage{rotating}
\usepackage{framed}
\usepackage{natbib}
\usepackage{url}

\DeclareMathOperator*{\argmin}{arg\,min}

\DeclareMathOperator*{\diag}{Diag}
\DeclareMathOperator*{\dom}{dom}

\def\A{{\mathcal A}}

\def\O{{\mathcal O}}

\newtheorem{lemma}{Lemma}[section]
\newtheorem{definition}{Definition}[section]

\newtheorem{theorem}{Theorem}[section]
\newtheorem{assumption}{Assumption}[section]

\newtheorem{remark}{Remark}[section]

\renewcommand{\thesection}{\arabic{section}}
\renewcommand{\thesubsection}{\arabic{section}.\arabic{subsection}.}

\title {An Efficient Semismooth Newton Method for Adaptive Sparse Signal Recovery Problems}
\author{Yanyun Ding\thanks{Department of Operations Research and Information Engineering, Beijing University of Technology, Beijing 100124, P.R. China (Email: dingyanyunhenu@163.com).}
\and
Haibin Zhang\thanks{Department of Operations Research and Information Engineering, Beijing University of Technology, Beijing 100124, P.R. China (Email: zhanghaibin@bjut.edu.cn).}
\and
Peili Li\thanks{School of Statistics, Key Laboratory of Advanced Theory and Application in Statistics and Data Science-MOE, East China Normal University, Shanghai 200062, P.R. China (Email: plli@sfs.ecnu.edu.cn).}
\and
Yunhai Xiao\thanks{School of Mathematics and Statistics,
Henan University, Kaifeng 475000, P.R. China (Email: yhxiao@henu.edu.cn).}
}
\date{\today}

\begin{document}
\maketitle

\begin{abstract}
We know that compressive sensing  can establish stable sparse recovery results from highly undersampled data under a restricted isometry property  condition. In reality, however,  numerous problems are coherent, and vast majority  conventional methods might work not so well.
Recently, it was shown that using the difference between $\ell_1$- and $\ell_2$-norm as a regularization always has superior performance.
In this paper, we propose an adaptive $\ell_p$-$\ell_{1-2}$ model where the $\ell_p$-norm with $p\geq 1$ measures the data fidelity and the $\ell_{1-2}$-term measures the sparsity.
This proposed model has the ability to deal with different types of noises and extract the sparse property even under high coherent condition.
We use a proximal majorization-minimization technique to handle the nonconvex regularization term and then employ a semismooth Newton method to solve the corresponding convex relaxation subproblem.
We prove that the sequence generated by the semismooth Newton method admits fast local convergence rate to the subproblem under some technical assumptions.
Finally, we do some numerical experiments to demonstrate the superiority of the proposed model and the progressiveness of the proposed algorithm.
\end{abstract}

\section{Introduction}\label{intro}
It is well known that compressive sensing (CS) is a novel signal acquisition paradigm to sense sparse signals by acquiring a set of incomplete or even noises polluted measurements, and  subsequently recovers the signals by solving an optimization problem.
Let $\underline{x}\in\mathbb{R}^n$ be a compressible or (approximately) sparse signal under a suitable basis (e.g.  Fourier or wavelet basis).
The main idea of CS is to firstly project $\underline{x}$ onto a certain subspace via a linear
operator $A$, i.e, $A\underline{x}=b$ with $b\in\mathbb{R}^m$ and
$m\ll n$, and then reconstruct $\underline{x}$ from the undersampled data $b$ via finding the sparest solution of the following underdetermined linear system:
$$
Ax=b.
$$
Besides, in the process of acquiring, storing, transmitting, or displaying,
the undersampled data may be inevitably degenerated by noise.
In this case, it is from the basic theory of CS \cite{Cand3,Donoho4,Don10} that, the task of recovering $\underline{x}$ can be characterized as
the following $\ell_1$-norm regularized least square
$$
   \min_{x\in\mathbb{R}^n} \frac{1}{2}\|Ax-b\|^2_2+\lambda\|x\|_1,
$$
where $\|\cdot\|_1$ is an $\ell_1$-norm function,
and $\lambda>0$ is a weighting parameter to balance both terms for minimization.
A deterministic result shows that it is possible to recover the original signal  by $\ell_1$-norm minimization when the number of nonzeros of $\underline{x}$ is less than $(1+1/\zeta)/2$ \cite{Don10,Gri11},
  where $\zeta$ is the so-called mutual coherence of matrix $A$.
This fact indicates that using the $\ell_1$-norm alone may not perform well for highly coherent matrices, and hence, some improvements are highly required.

A valid approach among such efforts is the using of the difference between the $\ell_1$- and $\ell_2$-norm as a regularization \cite{lou13,lou16,lou14,lou15}, which always has superior performance over the $\ell_1$-norm and even the $\ell_q$-norm with $q\in(0,1)$ alone \cite{chen17,chen18,chen20} in improving the sparsity when the sensing matrix $A$ is highly coherent.
Based on the fact, the recovering task of $\underline{x}$ using $\ell_{1}$-$\ell_{2}$-norm can be explicitly formulated as
\begin{equation}\label{l1l2}
   \min_{x\in\mathbb{R}^n} \frac{1}{2}\|Ax-b\|^2_2+\lambda(\|x\|_1-\|x\|_2),
\end{equation}
where ``$\|x\|_1-\|x\|_2$" (short for $\ell_{1-2}$) is called a regularization.
It should be noted that despite having some attractive features as shown by Yin et al. \cite{lou15}, the quality of the resolutions derived by model (\ref{l1l2}) heavily relies on the knowing of the standard deviation of the noise.
Fortunately, the square-root-loss estimator of Belloni et al. \cite{Bell21}, say $\|Ax-b\|_2$,  is proved to be achieving the near-oracle rates of convergence without knowing the standard deviation of  the noise under suitable design conditions\cite{bell22,cui23}.
The data fidelity with form $\|Ax-b\|_1$ has also been evidently shown to be more robust when the noises are not normal but heavy-tailed or heterogeneous, see e.g.,  \cite{bell24,lu26,wang25,xiu28}.
Besides,  the data fidelity with form $\|Ax-b\|_{\infty}$ is also known to be very suitable for dealing with the uniformly distributed noise and quantization error, see e.g., \cite{wen,xue27,zhang40}.
Therefore, a more natural question is whether or not one can design a more flexible and robust reconstruction model as well as an efficient algorithm which is capable of dealing with all the three types of noise mentioned above?

The main purpose of this paper is to answer the question positively in the sense of considering the following $\ell_p$-$\ell_{1-2}$  minimization problem
\begin{equation}\label{lp}
   \min_{x\in\mathbb{R}^n} \|Ax-b\|_p+\lambda(\|x\|_1-\beta\|x\|_2),
\end{equation}
where $\beta \geq 0$  and $\|\cdot\|_p$ is a $\ell_p$-norm function whose proximal mapping is assumed to be strongly semismooth, e.g., $p=1,2,\infty$ as stated in \cite[Remark 2]{lin38}.
The model (\ref{lp}) has many attractive features that, the $\|Ax-b\|_p$ data fidelity term makes it can deal with different types of noise if $p$ is chosen adaptively, and the $\ell_{1-2}$ regularization ensures the sparse property can be well extracted even under high coherent condition.
For example,  the choices of $p=1$, $2$ and $\infty$ is appropriate for log-normal noise or Laplace noise, Gaussian noise, and uniformly distributed noise, respectively.
Nevertheless, comparing with (\ref{l1l2}), it is much more difficult and challenging to solve due to the non-differentiability of the $\ell_p$-norm along with
the nonconvexity of the difference between the $\ell_1$-norm and $\ell_2$-norm.

Noting that  the ``$\|x\|_1-\beta\|x\|_2$" term is actually  a difference of two convex functions,
Yin et al. \cite{lou15} linearized the concave term ``$-\|x\|_2$"
at a given point to get a convex relaxation minimization problem, and then employed an alternating direction methods of multipliers on the resulting convex relaxation problem.
Beside, Tang et al. \cite{tang1} added a pair of proximal terms to the convex problems, and then
proposed an efficient semismooth Newton method.
Inspired from both works of  \cite{tang1,lou15}, in this paper, we introduce an efficient proximal majorization-minimization semismooth Newton method algorithm to solve the general problem (\ref{lp}). Our algorithm firstly linearizes the concave term in (\ref{lp}) at current iteration to
get a convex nonsmooth minimization, adds a pair of proximal terms to the objective function, and then uses an efficient semismooth Newton method to solve the involved semismooth nonlinear system to ensure faster local convergence rate. It should be noted that, the first contribution of this paper is that our model (\ref{lp}) covers the root-square-loss as a special case, which is more flexible than the model considered by Tang et al. \cite{tang1}.
Besides, we focus on a nonsmooth concave term``$-\|x\|_2$" but not a differentiable concave function as in \cite{tang1}. We must clarify that, the diversification of the $\ell_p$-norm in (\ref{lp}) may make the positive definiteness of generalized Hessian at the solution be violated, which poses fundamental challenges to the design of efficient algorithms.
The second contribution of this paper is to address this issue in the sense that, we employ the semismooth Newton method to solve the resulting subproblem with choices $p=1,2,\infty$, and prove that each limit point of the generated sequence converges to an optimal solution of the subproblem and admits fast local convergence rate  under some technical assumptions.
Finally, we do a series of numerical experiments which demonstrate that the superiority of the proposed model (\ref{lp}) is remarkable and the proposed algorithm is very efficient.

The remaining parts of this paper are organized as follows. In Section \ref{preli}, we summarize some
basic definitions or concepts for subsequent arithmetic design and numerical implementations.
In Section \ref{comp}, we give our motivation and construct our algorithm. Besides, the convergence result
for the proposed algorithm is also included.
Numerical experiments and performance comparisons are reported in
Section \ref{num4}. Finally, the paper is concluded with some remarks in Section \ref{con5} .
\section{Preliminaries}\label{preli}
Let $\mathbb{R}^{n}$ be an $n$-dimensional Euclidean space endowed with an inner product $\langle \cdot,\cdot \rangle $ and its induced norm $\|\cdot\|_2$, respectively.
Let $f:\mathbb{R}^{n}\rightarrow(-\infty,+\infty]$ be a closed proper convex function.
The effective domain of $f$, which we denote by $\text{dom}(f)$, is defined as $\text{dom}(f)=\{x| f(x)<+\infty\}$.
The directional derivative of $f$ at a point $x\in\text{dom} (f)$ along a direction $d\in \mathbb{R}^{n}$ is
defined as
$$
f^{\prime}(x ; d):=\lim_{\tau \downarrow 0}\frac{f(x+\tau d)-f(x)}{\tau}.
$$
We say $\bar{x}\in\text{dom} (f)$ is the d-stationary point of $f$, if it satisfies
$f^{\prime}(\bar{x} ; d)\geq 0$, $\forall d\in \mathbb{R}^n$.
The Fenchel conjugate of $f$ at $x^*$ is defined by $f^{*}(x^*):=\sup _{x \in \mathbb{R}^{n}}\{\langle x, x^*\rangle-f(x)\}$.

Denote $\Phi_{t f}(x)$ be the  Moreau envelope function \cite{moreau31,Yosida32} of $f$ with parameter $t >0$, i.e.,
 $$
\Phi_{t f}(x) :=\min _{y \in \mathbb{R}^{n}}\big\{f(y)+\frac{1}{2 t}\|y-x\|_2^{2}\big\}, \quad \forall x \in \mathbb{R}^{n}.
$$
The proximal mapping of $f$ with  $t >0$ is defined as
$$
\operatorname{Prox}_{tf}(x) :=\underset{y \in \mathbb{R}^{n}}{\operatorname{argmin}}\big\{f(y)+\frac{1}{2t}\|y-x\|_2^{2}\big\}, \quad \forall x \in \mathbb{R}^{n}.
$$
From \cite{Hiriart33,Lema34}, it is known that $\Phi_{t f}(x)$ is continuously differentiable and convex with its gradient being given by
$$\nabla \Phi_{t f}(x)=t^{-1}(x-\operatorname{Prox}_{t f}(x)), \quad \forall x \in \mathbb{R}^{n}.$$
Besides, from \cite[Theorem 35.1]{Rock35}, we have the following famous Moreau's identity theorem, which plays a key rule in the subsequent analysis
\begin{equation}\label{moreau}
\text{Prox}_{tf}(x)+t\text{Prox}_{f^{*}/t}(x/t)=x.
\end{equation}

Next we quickly review some basic concepts and definitions which are used frequently in the subsequent arithmetic developments and numerical implementations.
Semismoothness was originally introduced by Miffin \cite{Miff29} for functionals, and the concept was extended to vector valued functions by Qi and Sun \cite{qi30}.

\begin{definition}\label{def1}(semismoothness \cite{Miff29,qi30}).
Let $\Psi: \mathcal{O}\subseteq \mathbb{R}^n\rightarrow \mathbb{R}^m$ be a locally Lipschitz contimuous function and $\mathcal{K}: \mathcal{O}\rightrightarrows\mathbb{R}^{m\times n}$ be a nonempty and compact valued, upper-semicontinuous set-valued mapping on the open set $\mathcal{O}$.
It is said $\Psi$ to be semismooth at $\nu\in\mathcal{O}$ with respect to the set-valued mapping $\mathcal{K}$ if $\Psi$ is directionally differentiable at $\nu$ and for any $\Gamma\in\mathcal{K}(\nu+\Delta \nu)$ with $\Delta\nu\rightarrow 0$ such that
$$
\Psi(\nu+\Delta \nu)-\Psi(\nu)-\Gamma \Delta \nu=o(\|\Delta \nu\|_2).
$$
Besides, it is said $\Psi$ to be $\gamma$-order $(\gamma>0)$ (strongly, if $\gamma=1$ ) semismooth at $\nu$ with respect to $\mathcal{K}$ if $\Psi$ is semismooth at $\nu$ and for any $\Gamma \in \mathcal{K}(\nu+\Delta \nu)$ such that
$$
\Psi(\nu+\Delta \nu)-\Psi(\nu)-\Gamma \Delta \nu=O(\|\Delta \nu\|_2^{1+\gamma}).
$$
Moreover, we say that $\Psi$ is a semismooth $(\gamma$-order semismooth, strongly semismooth) function on $\mathcal{O}$ with respect to $\mathcal{K}$ if it is semismooth $(\gamma$-order semismooth, strongly semismooth) at every $\nu \in \mathcal{O}$ with respect to $\mathcal{K}$.
\end{definition}
\begin{theorem}\label{the1}(Rademacher's Theorem \cite{Rock36})
Suppose that $\Psi: \mathbb{R}^n\rightarrow \mathbb{R}^m$ is locally Lipschitz
continuous on an open set $\mathcal{O}\subseteq \mathbb{R}^n$. Then $\Psi$ is almost everywhere (in the sense of Lebesgue
measure) Fr\'{e}chet-differentiable in $\mathcal{O}$.
\end{theorem}

Let $\mathcal{D}_{\Psi}$ be the set of all points where $\Psi$ is Fr\'{e}chet-differentiable and $J\Psi(x)\in \mathbb{R}^{m\times n}$ be the Jacobian of $\Psi$ at $x\in\mathcal{D}_{\Psi}$. For any $x\in\mathcal{O}$, the B-subdifferential of $\Psi$ at $x$ is defined by
$$
\partial_{B} \Psi(x):=\{V \in \mathbb{R}^{m \times n} \mid \exists\{x^{k}\} \subseteq \mathcal{D}_{\Psi} \text { such that } \lim _{k \rightarrow \infty} x^{k}=x \text { and } \lim _{k \rightarrow \infty} J \Psi(x^{k})=V\}.
$$
The Clarke Jacobian of $\Psi$ at $x$ is defined as the convex hull of the B-subdifferential of $\Psi$ at $x$, that is $\partial \Psi(x):=\text{conv}(\partial_{B} \Psi(x))$.
For more detail, one may refer to \cite{Clarke37}.
Let $f$ be specially defined as $f:\mathbb{R}^{n}\rightarrow\mathbb{R}$. The Clarke subdifferential of $f$ at $x$ is defined as
$$
\partial f(x):=\big\{h \in \mathbb{R}^{n} \mid \limsup _{x^{\prime} \rightarrow x, t \downarrow 0} \frac{f(x^{\prime}+t y)-f(x^{\prime})-t h^{\top} y}{t} \geq 0, \forall y \in \mathbb{R}^{n}\big\}.
$$
The regular subdifferential of $f$ at $x$ is defined as
\begin{equation}\label{resubd}
\hat{\partial} f(x):=\big\{h \in \mathbb{R}^{n} \mid \liminf _{x \neq y \rightarrow x} \frac{f(y)-f(x)-h^{\top}(y-x)}{\|y-x\|_2} \geq 0\big\}.
\end{equation}

Let $\Psi: \mathbb{R}^n\rightarrow \mathbb{R}^n$ is a locally Lipschitz
continuous function. The generalized Newton's method for solving $\Psi(x)=0$ was firstly proposed by Kummer \cite{KUMMER}.
Correspondingly, one can refer to \cite[Theorem 3.1]{QI93} and \cite[Theorem 3.4]{qi30} for its local convergence analysis at the case of $\Psi$ being semismooth.
The  most distinctive feature of the inexact semismooth Newton method is to find an approximate solution $d^k$ to
$$
\Psi(x^k)+V^kd=0,
$$
such that
$$
\|\Psi(x^k)+V^kd^k\|\geq\eta\|\Psi(x^k)\|,
$$
where $V^k\in\partial\Psi(x^k)$ and $\eta>0$ is a small constant.
It was shown that when the elements in $\partial\Psi(\bar x)$ are nonsingluar and $\Psi$ is semismooth at solution $\bar x$, the generated sequence $\{x^k\}$ is well-defined and converges to $\bar x$ suplinearly and even quadratic for $k$ sufficiently large.
For more details of inexact semismooth Newton method, one can refer to \cite{FK97,MQ95} and the references therein.

Let $B^{(r)}_{p}$ be a $\ell_{p}$-norm ball with radius $r>0$ and define $\Pi_{B^{(r)}_{p}}(\cdot)$ be an orthogonal projection onto $B^{(r)}_{p}$.
Using this notation, we summarize some proximal mappings of norm functions, which play key rules in the algorithmic construction.
The following Lemma \ref{lemma11} reports some  applications of the Moreau's identity (\ref{moreau}) for some typical norm functions.
These results are known in optimization literature, hence, we omit its proof here.

\begin{lemma}\label{lemma11}
For any $x^*\in \mathbb{R}^n$, it holds that\\
(a) Let $f(x):=\mu\|x\|_1$ with $\mu>0$, then $f^{*}(x^*)=\delta_{B^{(\mu)}_{\infty}}(x^*)$ with ${B^{(\mu)}_{\infty}}(x^*):=\{x^* \ | \ \|x^*\|_{\infty}\leq \mu\}$ and
$$
\operatorname{Prox}_f(x^*)=x^*-\Pi_{B^{(\mu)}_{\infty}}(x^*) \quad \text{with} \quad
(\Pi_{B^{(\mu)}_{\infty}}(x^*))_i=\left\{\begin{array}{ll}
x^*_i, & \text { if }  |x^*_i|\leq\mu,\\
\operatorname{sign}(x^*_i) \mu, & \text { if }  |x^*_i|>\mu,
\end{array}\right.
$$
where $i=1,\ldots,n$ and $sign(\cdot)$ is a sign function of a vector. \\
(b) Let $f(x):=\mu\|x\|_2$ with $\mu>0$, then $f^{*}(x^*)=\delta_{B^{(\mu)}_{2}}(x^*)$ with ${B^{(\mu)}_{2}}(x^*):=\{x^* \ | \ \|x^*\|_{2}\leq \mu\}$ and
$$
\operatorname{Prox}_f(x^*)=x^*-\Pi_{B^{(\mu)}_{2}}(x^*)\quad \text{with} \quad
\Pi_{B^{(\mu)}_{2}}(x^*)=\left\{\begin{array}{ll}
x^*, & \text { if }  \|x^*\|_2\leq\mu,\\
\mu\frac{x^*}{\|x^*\|_2}, & \text { if }  \|x^*\|_2>\mu.
\end{array}\right.
$$
(c) \cite{lin38} Let $f(x):=\mu\|x\|_{\infty}$ with $\mu>0$, then $f^{*}(x^*)=\delta_{B^{(\mu)}_{1}}(x^*)$ with ${B^{(\mu)}_{1}}(x^*):=\{x^* \ | \ \|x^*\|_{1}\leq \mu\}$ and
$$
\operatorname{Prox}_f(x^*)=x^*-\Pi_{B^{(\mu)}_{1}}(x^*) \quad \text{with} \quad
\Pi_{B^{(\mu)}_{1}}(x^*)=\left\{\begin{array}{ll}
x^*, & \text { if }  \|x^*\|_1\leq\mu,\\
\mu P_{x^*} \Pi_{\Delta_{n}}\left(P_{x^*} x^* /\mu\right), & \text { if }  \|x^*\|_1>\mu.
\end{array}\right.
$$
where $P_{x^*}:=\diag(\operatorname{sign}(x^*))$ and $\Pi_{\Delta_{n}}(\cdot)$ denotes the projection onto the simplex $\Delta_{n}:=\{x \in \mathbb{R}^{n} \mid e_{n}^{\top} x=1, x \geq 0\}$, in which $\diag(\cdot)$ denotes a diagonal matrix with elements of a given vector on its diagonal positions.
\end{lemma}

The following Lemma \ref{lemma22} characterizes the Clarke Jacobian of proximal mapping operators regarding to some special norm functions.
The third part of this lemma is recently obtained by Li et al. \cite{lin38}.

\begin{lemma}\label{lemma22}
For any $\vartheta\in\mathbb{R}^n$, it holds that\\
(a) The Clarke Jacobian
 of $\operatorname{Prox}_{f}(\cdot)$ with $f(x):=\mu\|x\|_1$ at $\vartheta$ is given by
$$
\partial \operatorname{Prox}_{f}(\vartheta)
=\bigg\{\diag(\theta) \Big| \theta \in \mathbb{R}^{n}, \ \theta_{i} \in
\left\{\begin{array}{ll}
\{1\}, & \text { if }|\vartheta_{i}|>\mu, \\
{[}0,1{]}, & \text { if } |\vartheta_{i}|=\mu, \quad i=1, \ldots, n \\
\{0\}, & \text { if }|\vartheta_{i}|<\mu,
\end{array}
\right.
\bigg\}.
$$
(b) The  Clarke Jacobian of $\operatorname{Prox}_{f}(\cdot)$ with $f(x):=\mu\|x\|_2$ at $\vartheta$ is given by
$$
\partial \operatorname{Prox}_{f}(\vartheta)=\left\{\begin{array}{ll}
(1-\dfrac{\mu}{\|\vartheta\|_2}) \mathbf{I}_{n}+\mu\dfrac{\vartheta \vartheta^{\top}}{\|\vartheta\|_2^{3}}, & \text { if }\|\vartheta\|_2>\mu,\\[2mm]
\{\kappa\dfrac{\vartheta \vartheta^{\top}}{\mu^2}\mid 0\leq\kappa\leq1\}, & \text { if }\|\vartheta\|_2=\mu,\\[2mm]
\mathbf{0}_{n}, & \text { if }\|\vartheta\|_2<\mu,
\end{array}\right.
$$
where $\mathbf{I}_{n}$ and $\mathbf{0}_n$ are identify matrix and zero matrix in $n$-dimensional Euclidean space, respectively.\\
(c)\cite{lin38} The Clarke Jacobian of $\operatorname{Prox}_{f}(\cdot)$ with $f(x):=\mu\|x\|_\infty$ at $\vartheta$ is given by
$$
\partial \operatorname{Prox}_{f}(\vartheta)=\mathbf{I}_n-H,
$$
where $H\in\partial\Pi_{B^{(\mu)}_{1}}( \vartheta)$ with explicit form
$$
H=
\left\{\begin{array}{ll} P_{\vartheta} \tilde{H} P_{\vartheta}, & \text { if }\|\vartheta\|_1>\mu,\\
\mathbf{I}_n, & \text { if }\|\vartheta\|_1\leq\mu,
\end{array}\right.
$$
where $ \tilde{H}=\text{Diag}(r)-\frac{1}{nnz(r)} r r^{\top} \in \partial \Pi_{\Delta_{n}}(\vartheta)$ with $r \in \mathbb{R}^{n}$ defined as $r_{i}=1$ if $\left(\Pi_{\Delta_{n}}\left(P_{\vartheta} \vartheta / \mu\right)\right)_{i} \neq 0,$ and $r_{i}=0$
otherwise, and $\text{nnz}(\cdot)$ denotes the number of non-zeros elements of a vector.
\end{lemma}

It is noteworthy that,  we specifically formalize the element of $\operatorname{Prox}_{f}(\cdot)$ at each case for the purpose of catching the subsequent algorithm's development easily.
\begin{remark}\label{selec}
In our numerical experiments part, we choose
the element $\hat{\Theta}$ in $\partial \operatorname{Prox}_{\mu\|\cdot\|_{1}}(\vartheta)$ at a given point $\vartheta$ as
$$
\hat{\Theta}=\operatorname{diag}(\theta) \text { with } \theta_{i}=\left\{\begin{array}{ll}
1, & \text { if }\left|\vartheta_{i}\right|>\mu,\\[2mm]
0, & \text { if }\left|\vartheta_{i}\right| \leq \mu, \quad i=1, \ldots, n,
\end{array}\right.
$$
choose the element $\hat{\Theta}$ in $\partial \operatorname{Prox}_{\mu\|\cdot\|_{2}}(\vartheta)$ as
$$
\hat{\Theta}=\left\{
\begin{array}{ll}
(1-\dfrac{\mu}{\|\vartheta\|_2}) \mathbf{I}_{n}+\mu\dfrac{\vartheta \vartheta^{\top}}{\|\vartheta\|_2^{3}}, & \text { if }\|\vartheta\|_2>\mu,\\
\mathbf{0}_n, & \text { if }\|\vartheta\|_2\leq\mu,
\end{array}\right.
$$
and choose the element $\hat{\Theta}$ in $\partial \operatorname{Prox}_{\mu\|\cdot\|_{\infty}}(\vartheta)$ as
$$
\hat{\Theta}=\left\{
\begin{array}{ll} \mathbf{I}_n-P_{\vartheta} \tilde{H} P_{\vartheta}, & \text { if }\|\vartheta\|_1>\mu,\\[2mm]
\mathbf{0}_n, & \text { if }\|\vartheta\|_1\leq\mu,
\end{array}\right.
$$
where $\tilde H$ is defined in (c) of Lemma \ref{lemma22}.
\end{remark}

\section{Computational approach}\label{comp}
This section is devoted to constructing a new algorithm to solve the model (\ref{lp}).
Similar to \cite{tang1}, the basic idea of our algorithm also uses a proximal majorization technique to handle the noncovex term  in the objective function of (\ref{lp}).
\subsection{Proximal majorization-minimization algorithm}

The difference between the $\ell_1$-norm and the $\ell_2$-norm causes the main difficulties because it is not only nonsmoothness but also nonconvexity.
To overcome this obstacle, as usual in optimization literature, e.g., \cite{lou15}, we linearize the non-convex term ``$-\|x\|_2$" at a given point $x^k$ to obtain the following convex relaxation minimization problem
$$
  \min_{x\in \mathbb{R}^n}\|Ax-b\|_p+\lambda \|x\|_1-\lambda\beta \big(\|x^k\|_2+\langle v^k, x-x^k \rangle\big),
$$
where $v^k\in \partial \|x^k\|_2$ is a subgradient. Omitting the constant term, it can be simplified as the following minimization problem
\begin{equation}\label{PMM11}
  \min_{x\in \mathbb{R}^n}\|Ax-b\|_p+\lambda \big(\|x\|_1-\beta \langle v^k, x \rangle\big).
\end{equation}
Noting that $\partial \| x\|_2=B_2^{(1)}$ if $ x=0$ and $ x/\| x\|_2$ otherwise. With a given $x^k$, it is reasonable to derive its solution $x^{k+1}$ via
solving the following convex minimization problem iteratively
\begin{equation}\label{dc}
x^{k+1}:=\left\{\begin{array}{ll}
\arg \min\limits _{x \in \mathbb{R}^{n}} \Big\{\|A x-b\|_{p}+\lambda\|x\|_{1}\Big\},  &\text{if} \ x^k=0,\\[3mm]
\arg \min\limits_{x \in \mathbb{R}^{n}}  \Big\{\|A x-b\|_{p}+\lambda\big(\|x\|_{1}-\beta\langle x^k/\|x^k\|_{2}, x \rangle\big)\Big\}, &\text{otherwise}.
\end{array}\right.
\end{equation}
Comparing to (\ref{lp}), both approximated problems in (\ref{dc}) are convex, and hence they can be solved effectively via the alternating direction method of multiplier despite a pair of nonsmooth terms contained. The details of difference of convex functions algorithm will be shown in the Section \ref{num4}.

In this paper, we focus on the employing of the  proximal majorization-minimization semismooth Newton method to solve the nonsmooth convex minimization (\ref{PMM11}).
Based on the idea of Tang et al. \cite{tang1}, we also add a pair of proximal terms to the objective function of (\ref{PMM11}) to get the following proximal majorized minimization problem
\begin{equation}\label{PMMmodel}
\min_{x\in \mathbb{R}^n}\Big\{\tilde{f}^k(x; \sigma^k, \tau^k, v^k, x^k, \tilde{b}^k):=\|Ax-b\|_p+\lambda\big(\|x\|_1-\beta \langle v^k, x \rangle\big)+\frac{\sigma^k}{2}\|x-x^k\|_2^2+\frac{\tau^k}{2}\|Ax-\tilde{b}^k\|_2^2\Big\},
\end{equation}
where $\tilde{b}^k\in \mathbb{R}^m$ is a given data defined later, and $\sigma^k>0$, $\tau^k>0$ are given parameters.
For convenience, we sometimes abbreviate $\tilde{f}(x; \sigma^k, \tau^k, v^k, x^k, Ax^k)$ as $\tilde{f}^k(x)$.
The iterative framework of the proximal majorization-minimization algorithm for solving the $\ell_p$-$\ell_{1-2}$-regularized minimization problem (\ref{lp}) is summarized as follows:

\begin{framed}
{\bf Algorithm: PMM\_$\ell_p$-$\ell_{1-2}$}
\vskip 1.0mm \hrule \vskip 1mm
\begin{itemize}
\item[Step 0.] Input $p:=1$, $2$, or $\infty$; input positive constants $\sigma^{0}$, $\tau^{0}$, and $\rho\in(0,1)$. Let $k:=0$.
\item[Step 1.] Compute an initial point via:
\begin{align}
 x^0:&=\arg\min_{x\in \mathbb{R}^n}\tilde{f}^0(x; \sigma^0, \tau^0, 0, 0, b)\nonumber\\
 &= \arg\min_{x\in \mathbb{R}^n}\|Ax-b\|_p+\lambda \|x\|_1+\frac{\sigma^0}{2}\|x\|_2^2+\frac{\tau^0}{2}\|Ax-b\|_2^2.\label{PMM1}
\end{align}
\item[Step 2.] Main loop:
\begin{itemize}
\item[Step 2.1] Terminate if some stopping criterions satisfied, output $x^k$. Otherwise, continue.

\item[Step 2.2] Set $v^k:=0$ if $x^k=0$. Otherwise set $v^k:=x^k/\|x^k\|_2$.

\item[Step 2.3]Compute
 \begin{align}
 x^{k+1}:&=\arg\min_{x\in \mathbb{R}^n}\Big\{\tilde{f}^k(x; \sigma^k, \tau^k, v^k, x^k, Ax^k)+\langle \delta^k, x-x^k\rangle\Big\}\nonumber\\
 &= \arg\min_{x\in \mathbb{R}^n}\bigg\{
 \begin{array}{c}
 \|Ax-b\|_p+\lambda\big(\|x\|_1-\beta \langle v^k, x \rangle\big)+\frac{\sigma^{k}}{2}\|x-x^k\|_2^2\label{PMM2}\\[2mm]
 +\frac{\tau^{k}}{2}\|Ax-Ax^k\|_2^2+\langle \delta^k, x-x^k\rangle
 \end{array}
 \bigg\}
 \end{align}
via solving its dual problem. That is, applying Algorithm SSN in Subsection \ref{sub33} to find an approximate solution $u^{j+1}$ of (\ref{nabla}) such that the error vector $\delta^k$ satisfies the following accuracy condition
\begin{equation}\label{stopping}
  \|\delta^k\|_2\leq \frac{\sigma^{k}}{4}\|x^{k+1}-x^k\|_2+\frac{\tau^{k}\|Ax^{k+1}-Ax^k\|^2_2}{2\|x^{k+1}-x^k\|_2},
\end{equation}
where
\begin{equation}\label{xk11}
x^{k+1}:=\text{Prox}_{\sigma^{-1}\lambda \| \cdot \|_1}\big(x^k+\sigma^{-1}(\lambda\beta v^k-A^{\top}u^{j+1})\big).
\end{equation}
\end{itemize}
\item[Step 3.]  Update $\sigma^{k+1}:=\rho \sigma^{k}$ and $\tau^{k+1}:=\rho \tau^{k}$. Let $k:=k+1$ and go to Step 2.
\end{itemize}
\end{framed}

It is clear to see that the given $\tilde{b}^k$ in (\ref{PMMmodel}) is actually undersampled data $b$ at the initial step, and it is chosen as $Ax^k$ at the iterative process.
We should note that the subproblem (\ref{PMM1}) can be solved exactly or inexactly via any algorithms because it only used to determine an initial point for the following main loops.
For simplicity, we omit the details on the solving of (\ref{PMM1}) because it fills into the framework of (\ref{PMM2}).
What's more important is to solve the subproblem (\ref{PMM2}) by applying a semismooth Newton method from the aspect of duality, that is, getting an approximate solution $u^{j+1}$ by Algorithm SSN described later and then setting $x^{k+1}$ via (\ref{xk11}).
It should be emphasized that the error vector $\delta^k$ in (\ref{PMM2}) cannot be explicitly predetermined, but here it is used to show that if an inexact solution $x^{k+1}$ from $u^{j+1}$ is computed, there must be a $\delta^k$ to satisfy (\ref{stopping}).
The appearances of the error vector in (\ref{PMM2}) is actually to show that the problem (\ref{PMMmodel}) has been solved approximately.
At last but not at least, the proximal mapping in (\ref{xk11}) is easily implementable because Lemma \ref{lemma11} ensures its explicit solutions for $\ell_1$-norm function.

\subsection{Convergence analysis}
In this section, we establish the convergence result of the sequence generated by Algorithm PMM\_$\ell_p$-$\ell_{1-2}$.
The following proof process can be considered as a generation of \cite[Theorem 15]{tang1} in the sense of using $\ell_p$-norm functions.

\begin{theorem}\label{thelim}
Denote
$$
f_1(\sigma, \tau):= \underset{x \in \mathbb{R}^{n}}{\min}\|Ax-b\|_p+\lambda \|x\|_1+\frac{\sigma}{2}\|x\|_2^2+\frac{\tau}{2}\|Ax-b\|_2^2.
$$
Then we have
$$
\lim _{\sigma, \tau \rightarrow 0}f_1(\sigma, \tau)=\min _{x \in \mathbb{R}^{n}}\|Ax-b\|_{p}+\lambda \|x\|_1.
$$
\end{theorem}
\begin{proof}
For any $\sigma$, $\tau>0$ and $x\in \mathbb{R}^n$, we have
$$
f_1(\sigma, \tau) \leq\|Ax-b\|_{p}+\lambda \|x\|_1+\frac{\sigma}{2}\|x\|^{2}_2+\frac{\tau}{2}\|Ax-b\|^{2}_2.
$$
Taking limits on both hand-sides of this inequality as $\sigma, \tau \rightarrow 0$, it gets that
$$
\lim \limits_{\sigma, \tau \rightarrow 0} f_1(\sigma, \tau) \leq\|Ax-b\|_{p}+\lambda \|x\|_1, \ \forall \ x\in\mathbb{R}^n,
$$
which means
$$
\lim _{\sigma, \tau \rightarrow 0} f_1(\sigma, \tau)\leq\min _{x \in \mathbb{R}^{n}} \|Ax-b\|_{p}+\lambda \|x\|_1.
$$
On the other hand, the definition indicates that
$$
f_1(\sigma, \tau)\geq\min _{x \in \mathbb{R}^{n}} \|Ax-b\|_{p}+\lambda \|x\|_1.
$$
Hence, the desired result follows.
\end{proof}

Theorem \ref{thelim} indicates that the minimum value of the function located at the right-hand side of  (\ref{PMM1}) is close to the optimal objective value of (\ref{lp}) if $\beta\equiv0$ and $\sigma^0$, $\tau^0$ are sufficiently small.
Based on this equivalence, we are ready to establish the convergence result of Algorithm PMM\_$\ell_p$-$\ell_{1-2}$ under the following assumption.
\begin{assumption}\label{assu1}
For all $ \lambda>0$, it holds that
$$
f(x):=\|Ax-b\|_{p}+\lambda (\|x\|_1-\beta \|x\|_2)\rightarrow \infty
$$
as $x \rightarrow \infty$, that is to say $f(x)$ is coercive in the sense that the lower level set $\{x\in \mathbb{R}^n\mid f(x)\leq f(x^0)\}$ is bounded if $f(x^0)$ is bounded $\forall x^0\in \mathbb{R}^n$.
\end{assumption}

It should be noted that this assumption holds automatically if $\beta\in[0,1]$ because of the nonnegative of the difference of $(\|x\|_1-\beta \|x\|_2)$ at this case.
Using this notation, the $x$-subproblem at step 2.3 can be rewritten as
$$
x^{k+1}= \arg\min_{x\in \mathbb{R}^n} \tilde{f}^{k}(x)+\langle\delta^{k}, x-x^{k}\rangle.
$$

\begin{lemma}\label{lemma2}
 Let $x^{k+1}$ be an optimal solution of the subproblem in (\ref{PMM2}) such that the condition (\ref{stopping}) holds. Then we have
 $$
\tilde{f}^{k}(x^{k}) \geq \tilde{f}^{k}(x^{k+1})-\frac{\sigma^{ k}}{4}\|x^{k+1}-x^{k}\|^{2}_2-\frac{\tau^{k}}{2}\|A x^{k+1}-A x^{k}\|^{2}_2.
$$
\end{lemma}
\begin{proof}
Since $\tilde{f}^{k}(\cdot)$ is a convex function and $-\delta^{k} \in \partial \tilde{f}^{k}(x^{k+1})$ from the first-order condition of (\ref{PMM2}), then we have
\begin{align*}
\tilde{f}^{k}(x^{k})-\tilde{f}^{k}(x^{k+1}) &\geq\langle\delta^{k}, x^{k+1}-x^{k}\rangle \geq - \|\delta^{k}\|_2 \|x^{k+1}-x^{k}\|_2\\
&\geq-\frac{\sigma^{k}}{4}\|x^{k+1}-x^{k}\|^{2}_2-\frac{\tau^{ k}}{2}\|A x^{k+1}-A x^{k}\|^{2}_2,
\end{align*}
where the first inequality is from the definition of subdifferentiability and the last inequality is from the condition (\ref{stopping}).
\end{proof}

\begin{lemma}\label{lemma3}
Denote $f(x):=\|Ax-b\|_p+\lambda(\|x\|_1-\beta\|x\|_2)$. The vector $\bar{x} \in int\dom(f)$ is a d-stationary point of (\ref{lp}) if and only if there exist $\sigma$, $\tau \geq 0$ such that
$$
\bar{x} \in \underset{x \in \mathbb{R}^{n}}{\operatorname{argmin}} \ \tilde{f}(x ; \sigma, \tau, \bar{v}, \bar{x}, A \bar{x}),
$$
where $\bar{v}\in \partial \|\bar{x}\|_2$.
\end{lemma}
\begin{proof}
At the beginning of the proof, we use this abbreviation $\tilde{f}^-(x):=\tilde{f}(x ; \sigma, \tau, \bar{v}, \bar{x}, A \bar{x})$ for convenience.
Noting that $f(\cdot)$ is locally Lipschitz continuous near $\bar{x}$ and directionally differentiable
at $\bar{x}$, then $0 \in \hat{\partial} f(\bar{x})$
is equivalent to $\bar{x}$ being a d-stationary point of $f$, where $\hat{\partial}f(\cdot)$ is the regular subdifferentiability defined in (\ref{resubd}).
It is not difficult to see that $\hat{\partial} f(\bar{x})=\partial\tilde{f}^-(\bar{x})$ from the definition of $f(\cdot)$ and $\tilde{f}^-(\cdot)$, where $\bar{v}\in \partial \|\bar{x}\|_2$. For given $\sigma, \tau$ and $\bar{x},$ the function
$\tilde{f}^-({x})$ is convex, thus $0 \in \partial \tilde{f}^-(\bar{x})$ is equivalent to
$\bar{x} \in \operatorname{argmin}_{x \in \mathbb{R}^{n}}\tilde{f}^-(x)$, which indicates the desired result of this lemma.
\end{proof}

The following theorem shows that the sequence $\{x^k\}$ generated by Algorithm PMM\_$\ell_p$-$\ell_{1-2}$ is convergent.
\begin{theorem}\label{the2}
Suppose that the function $f(x)$ in (\ref{lp}) is bounded and the Assumption \ref{assu1} holds.
 Assume that $\{\sigma^{k}\}$ and $\{\tau^{k}\}$ be the sequences converge to zero. Let $\{x^{k}\}$ be the sequence generated by Algorithm PMM\_$\ell_p$-$\ell_{1-2}$, we have \\
(a) $f(x^k)-f(x^{k+1})\geq \frac{\sigma^{k}}{4}\|x^{k+1}-x^k\|_2^2\geq 0$.\\
(b) $\{x^k\}$ is bounded if $f(x^0)$ is bounded.\\
(c) $\|x^{k+1}-x^k\|_2\rightarrow 0$ as $k\rightarrow \infty$.\\
(d) Any limit point $x^{\infty}$ of the sequence $\{x^{k}\}$  is a d-stationary point of (\ref{lp}).
\end{theorem}
\begin{proof}
(a) From step 2.3 in  Algorithm PMM\_$\ell_p$-$\ell_{1-2}$, it is easily known that $f(x^{k})=\tilde{f}^{k}(x^{k})$.
Combing with Lemma \ref{lemma2}, we have
\begin{align*}
f(x^{k}) =&\tilde{f}^{k}(x^{k}) \geq \tilde{f}^{k}(x^{k+1})-\frac{\sigma^{ k}}{4}\|x^{k+1}-x^{k}\|^{2}_2-\frac{\tau^{ k}}{2}\|A x^{k+1}-A x^{k}\|^{2}_2 \\
 \geq & f(x^{k+1})+\frac{\sigma^{ k}}{4}\|x^{k+1}-x^{k}\|^{2}_2,
\end{align*}
where the last inequality is from  the convexity of $\|x\|_2$.\\
(b) The assertion (a) indicates that $\{f(x^k)\}$ is monotonically decreasing,  which in turn shows that
$\{x^{k}\} \subseteq\{x \in \mathbb{R}^{n}: f(x) \leq f(x^{0})\}$. Hence, the sequence $\{x^k\}$ is bounded from Assumption \ref{assu1}.\\
(c) Since $f(x)$ is bounded and $\{f(x^k)\}$ is monotonically decreasing, the sequence $\{f(x^k)\}$ is convergent.  Then from the assertion (a), it indicates that the sequence  $\{\|x^{k+1}-x^k\|_2\}$ converges to zero.\\
(d) Let $x^{\infty}$ be a limit point of the subsequence $\{ x^{k}\}_{k \in \mathcal{K}} .$ We can easily prove that $\{x^{k+1}\}_{k \in \mathcal{K}}$ also converges to $x^{\infty} .$ It follows from the definition of $x^{k+1}$ that
$$
\tilde{f}^{k}(x) \geq \tilde{f}^{k}(x^{k+1})+\langle\delta^{k}, x^{k+1}-x\rangle \geq \tilde{f}^{k}(x^{k+1})-\|\delta^{k}\|_2\|x^{k+1}-x\|_2, \quad \forall x \in \mathbb{R}^{m}.
$$
Letting $k(\in \mathcal{K}) \rightarrow \infty,$ we obtain that
$$
x^{\infty} \in \underset{x \in \mathbb{R}^{n}}{\operatorname{argmin}} \ \tilde{f}(x ; \sigma^{\infty}, \tau^{ \infty}, v^{\infty}, x^{\infty},  A x^{\infty}),
$$
where $v^{\infty} \in \partial\|x^{\infty}\|_2,$ $\sigma^{\infty}=\lim _{k \rightarrow \infty} \sigma^{ k} \geq 0$ and $\tau^{\infty}=\lim _{k \rightarrow \infty} \tau^{k} \geq 0 $.
The desired result follows from Lemma
\ref{lemma3}.
\end{proof}

\subsection{Solving subproblem}\label{sub33}

The main computing burden in Algorithm PMM\_$\ell_p$-$\ell_{1-2}$  lies in solving the subproblems (\ref{PMM1}) and (\ref{PMM2}).
Fortunately, we will show that each subproblem can be solved effectively by the using of the semismooth Newton method which has been widely shown to have
faster local convergence rate.
Noting that (\ref{PMM1}) and (\ref{PMM2}) are actually in the framework of (\ref{PMMmodel}), hence we only focus on the algorithm's construction for problem (\ref{PMMmodel})

For the purpose the algorithm's design, we let $y:=Ax-b$, then  (\ref{PMMmodel}) is equivalent to
\begin{equation}\label{PMMmodel2}
\begin{array}{ll}
\min\limits_{x,y} & \|y\|_p+\lambda \|x\|_1-\lambda\beta \langle v^k, x \rangle+\frac{\sigma}{2}\|x-x^k\|_2^2+\frac{\tau}{2}\|y+b-\tilde{b}^k\|_2^2  \\
\text{s.t.} & Ax-y=b,
\end{array}
\end{equation}
where $\tilde{b}^k$ is actually $Ax^k$ at (\ref{PMM2}).
The Lagrangian function associated with problem (\ref{PMMmodel2})  is given by
\begin{align*}
\mathcal{L}(x,y;u)&= \|y\|_p+\lambda \|x\|_1-\lambda\beta \langle v^k, x \rangle+\frac{\sigma}{2}\|x-x^k\|_2^2+\frac{\tau}{2}\|y+b-\tilde{b}^k\|_2^2+\langle u,Ax-y-b\rangle,
\end{align*}
where $u\in\mathbb{R}^{m}$ is  a multiplier associated with the constraint.
The Lagrangian dual function $D(u)$ is defined as the minimum value of the Lagrangian function over $(x,y)$, that is
\begin{align*}
   D(u)
  =& \inf_{x,y} \mathcal{L}(x,y;u)\\[1mm]
  =& \min_{x} \Big\{\lambda \|x\|_1-\lambda\beta \langle v^k, x \rangle+\frac{\sigma}{2}\|x-x^k\|_2^2+\langle u,Ax\rangle\Big\}+\min_{y} \Big\{\|y\|_{p}+\frac{\tau}{2}\|y+b-\tilde{b}^k\|_2^2-\langle u,y\rangle\Big\}-\langle u,b\rangle,
\end{align*}
From the first-order optimality conditions of the $x$- and $y$-subproblems, it yields that its minimizers can be expressed explicitly as
\begin{align}
\bar{x}=\arg\min_{x} \big\{\lambda \|x\|_1-\lambda\beta \langle v^k, x \rangle+\frac{\sigma}{2}\|x-x^k\|_2^2+\langle u,Ax\rangle\big\}
=\text{Prox}_{\sigma^{-1}\lambda \| \cdot \|_1}\big(x^k+\sigma^{-1}(\lambda\beta v^k-A^{\top}u)\big),\label{xbar}
\end{align}
and
\begin{align}
\bar{y}=\arg\min_{y} \big\{\|y\|_{p}+\frac{\tau}{2}\|y+b-\tilde{b}^k\|_2^2-\langle u,y\rangle\big\}
=\text{Prox}_{\tau^{-1} \| \cdot \|_{p}}\big( \tau^{-1}u-b+\tilde{b}^k \big).\label{ybar}
\end{align}
The Lagrangian dual problem of (\ref{PMMmodel2}) is the maximizing this dual function $D(u)$ subject to some constraints, which can  be equivalently formulated as the following optimization problem by the using of the Moreau's identity theorem  (\ref{moreau}) and the equalities (\ref{xbar}) and (\ref{ybar}), that is
\begin{equation}\label{dual}
\min_{u} \Theta(u):=\left\{
\begin{array}{lll}
&\langle u, b\rangle +\dfrac{\tau}{2}\Big\|\tau^{-1}u-b+\tilde{b}^k\Big\|_2^2-\Big\|\text{Prox}_{\tau^{-1} \| \cdot \|_{p}}\big( \tau^{-1}u-b+\tilde{b}^k\big)\Big\|_p\\[2mm]
&-\dfrac{1}{2\tau}\Big\|\text{Prox}_{\tau \| \cdot \|^*_{p}}\big( u-\tau(b-\tilde{b}^k)\big)\Big\|_2^2+\dfrac{\sigma}{2}\Big\|x^k+\sigma^{-1}(\lambda\beta v^k-
A^{\top}u)\Big\|_2^2\\[2mm]
&-\lambda\Big\|\text{Prox}_{\sigma^{-1}\lambda \| \cdot \|_1}\big(x^k+\sigma^{-1}(\lambda\beta v^k-A^{\top}u)\big)\Big\|_1
-\dfrac{1}{2\sigma}\Big\|\text{Prox}_{\sigma(\lambda \| \cdot \|_1)^*}\big(\sigma x^k+\lambda\beta v^k-A^{\top}u\big)\Big\|^2_2
\end{array}
\right\}.
\end{equation}

From the equalities of (\ref{xbar}) and (\ref{ybar}), we can clearly observe that if the value of dual variable $u$ is known, the values of $x$ and $y$ can be obtained in a compact form.
It is from \cite[Theorem 31.5]{Rock35} that $\Theta(u)$ is first-order continuously
differentiable with gradient
$$
\nabla \Theta(u)=b+\text{Prox}_{\tau^{-1} \| \cdot \|_{p}}\Big( \tilde{b}^k-b+\tau^{-1}u\Big)-A\text{Prox}_{\sigma^{-1}\lambda \| \cdot \|_1}\Big(x^k+\sigma^{-1}(\lambda\beta v^k-A^{\top}u)\Big).
$$
Since $\Theta(u)$ is strongly convex, then its minimizer $\bar{u}$  can be obtained by the following first-order optimality condition
\begin{equation}\label{nabla}
  \nabla \Theta(u)=0.
\end{equation}
It is well known in optimization literature that $\text{Prox}_{\| \cdot \|_p}(\cdot)$ with $p=1$, $2$ and $\infty$ is strongly semismooth,  so do $\nabla \Theta(u)$,
which means that the Hessian of $\Theta(u)$ is unavailable. Fortunately, due to the special structures of the proximal mappings involved in $\nabla\Theta(u)$, the generalized Hessian can be obtained explicitly from Lemma \ref{lemma22}. Therefore, the semismooth Newton method of \cite{FK97,MQ95} can be employed.

We now focus on the applying of an efficient semismooth Newton  method to find an approximate solution of (\ref{nabla}).
Noting that the proximal mapping operators $\text{Prox}_{\tau^{-1} \| \cdot \|_{p}}(\cdot)$ and $\text{Prox}_{\sigma^{-1}\lambda \| \cdot \|_1}(\cdot)$ are Lipschitz continuous, then the following multifunction is well defined:
\begin{equation}\label{partt}
\tilde{\partial}^2\Theta(u):=\sigma^{-1}A\partial\text{Prox}_{\sigma^{-1}\lambda \| \cdot \|_1}\Big(x^k+\sigma^{-1}(\lambda\beta v^k-A^{\top}u)\Big)A^{\top}+\tau^{-1}\partial\text{Prox}_{\tau^{-1} \| \cdot \|_{p}}\Big( \tilde{b}^k-b+\tau^{-1}u\Big),
\end{equation}
where $\tilde{\partial}^2$ is the so-called generalized Hessian of $\Theta(\cdot)$, and $\partial(\cdot)$ is a Clarke Jacobian \cite{Clarke37} reviewed in Section \ref{preli}. Choose
$$
U\in \partial\text{Prox}_{\sigma^{-1}\lambda \| \cdot \|_1}\big(x^k+\sigma^{-1}(\lambda\beta v^k-A^{\top}u)\big) \quad \text{and} \quad V\in\partial\text{Prox}_{\tau^{-1} \| \cdot \|_{p}}\big( \tilde{b}^k-b+\tau^{-1}u\big),
$$
and let $H:=\sigma^{-1}AUA^{\top}+\tau^{-1}V$, then we have $H\in\tilde{\partial}^2\Theta(u)$.

In light of the above analysis, we list the framework of the semismooth Newton method to solve (\ref{PMMmodel}) as follows.
For details on the semismooth Newton method, one can refer to the excellent paper \cite{li38} and the references therein.

\begin{framed}
\noindent
{\bf Algorithm SSN: Semismooth Newton method for (\ref{nabla})}
\vskip 1.0mm \hrule \vskip 1mm
\begin{itemize}
\item[{Step 0.}] Given $\sigma>0$, $\tau>0$, $x^k\in\mathbb{R}^n$, $v^k\in \partial \|x^k\|_2$, and $\tilde{b}^k\in\mathbb{R}^m$, $p:=1$, $2$, or $\infty$;
input $\mu\in (0,1/2)$, $\bar{\eta}\in(0,1)$,
$\varrho\in(0,1]$, $\delta\in(0,1)$ and $u^0\in\mathbb{R}^m$. For $j=0,1,\ldots$, do the following steps iteratively:
\item[{Step 1.}] Select an element  $U^j\in \partial\text{Prox}_{\sigma^{-1}\lambda \| \cdot \|_1}\big(x^k+\sigma^{-1}(\lambda\beta v^k-A^{\top}u^j)\big)$ and $V^j\in\partial\text{Prox}_{\tau^{-1} \| \cdot \|_{p}}\big( \tilde{b}^k-b+$
 $\tau^{-1}u^j\big)$, and set $H^j:=\sigma^{-1}AU^j A^{\top}+\tau^{-1}V^j$. Employ a numerical method (e.g., conjugate gradient  method) to compute an approximate solution $\Delta u^j$ to the linear system
$$
  H^j\Delta u^j+\nabla \Theta (u^j)=0
$$
such that
\begin{equation*}
  \| H^j\Delta u^j+\nabla \Theta (u^j)\|_2\leq \min \{\bar{\eta}, \|\nabla \Theta (u^j)\|_2^{1+\varrho}\}.
\end{equation*}
\item[{Step 2.}] Find $\alpha_j:=\delta^{t_j}$, where $t_j$ is the first nonnegative integer $t$ such that
$$
\Theta (u^j+\delta^{t}\Delta u^j)\leq \Theta (u^j)+\mu \delta^{t}\langle \nabla \Theta (u^j), \Delta u^j\rangle.
$$
\item[{Step 3.}]Set $u^{j+1}:=u^j+\alpha_j\Delta u^j$.
\end{itemize}
\vspace{-.3cm}
\end{framed}

We now give a couple of remarks to make the algorithm easier to follow.
Firstly, the iterative process should be terminated until $\nabla\Theta(u^{j+1})$ satisfies a prescribed stopping criterion, e.g., $\|\nabla\Theta(u^{j+1})\|_2\leq\epsilon$ with a tolerance $\epsilon$. Then the $x^{k+1}$ getting from (\ref{xk11}) by using $u^{j+1}$ ensures that there must be an error vector $\delta^k$ to satisfy (\ref{stopping}), which indicates that $x^{k+1}$ is actually an inexact solution of (\ref{PMMmodel}).
Secondly, it will be shown from the following Theorem \ref{the33} that, under Assumptions \ref{assum32} and \ref{assum33}, the generalized Hessian $H\in\tilde{\partial}^2\Theta(\bar u)$ chosen according to Remark \ref{selec} is symmetric and positive definite, which means that the approximate solution  $\Delta u^j$ is well-defined.

To end this section, we report the convergence result of the Algorithm SSN with a sketched proof draw from \cite{li38}.
For this purpose, we need a pair of additional conditions listed below:

\begin{assumption}\label{assum32}
Suppose that the unique optimal solution $\bar{x}$ of the problem (\ref{PMMmodel}) satisfies $A\bar x-b\neq 0$, and specifically
$(A\bar{x}-b)_i\neq 0$ for each $i\in\{1,2, \ldots,n\}$ when $p=1$.
\end{assumption}

\begin{assumption}\label{assum33}
Let $\bar{u}$ be the unique solution of the problem (\ref{dual}), and denote $\hat{u}:= \tau^{-1}\bar{u}-b+\tilde{b}^k$.
Assume that the maximum eigenvalue $\lambda_{\max}(H)$ of the matrix $H\in\partial\Pi_{B^{(1/\tau)}_{1}}( \hat{u})$ with explicit form $H=P_{\hat{u}} \tilde{H} P_{\hat{u}}$, satisfies that $\lambda_{\max}(H)<1$.
\end{assumption}

{\color{red}It} should be noted that the second assumption can be met at some special cases. For example, if $\|\hat{u}\|_1>1/\tau$ and that there
only one entry of the projection of $P_{\hat u}\hat u\tau$ onto the simplex  $\Delta_n$  is not zero, i.e., $(\Delta_n(P_{\hat u}\hat u\tau))_i\neq 0$ with only one index $i$.
At this case, we have $r_i=1$ and $r_j=0$ for $j\neq i$. Using the notation $\tilde{H}=\text{Diag}(r)-rr^{\top}/\text{nnz(r)}$ and noting the definition of $P_{\hat{u}}$ in Lemma \ref{lemma22}, we can get that $H={\bf 0}_n$, and hence this assumption is satisfied.
Taking another example, if $(\Delta_n(P_{\hat u}\hat u\tau))_i\neq 0$ and $(\Delta_n(P_{\hat u}\hat u\tau))_j\neq 0$ with $\hat{u}_i\neq 0$ and $\hat{u}_j=0$.
From Lemma \ref{lemma22}(c) and the relation $\partial \text{Prox}_{\tau^{-1} \| \cdot \|_{\infty}}(\hat{u})={\bf I_n}-H$, it is not hard to deduce that
$H=\diag(0,\ldots,1/2,\ldots,0)$, i.e., a diagonal matrix with entry $1/2$ at the $k$-position and $0$ elsewhere.
In this instance, we also have $\lambda_{\max}(H)<1$.

\begin{theorem}\label{the33}
Suppose that the Assumptions \ref{assum32} and \ref{assum33} hold. Then the sequence $\{u^j\}$ generated by Algorithm
SSN converges to the unique optimal solution $\bar{u}$ of the problem (\ref{dual}) and satisfies
$$
\|u^{j+1}-\bar{u}\|_2= \O(\|u^j-\bar{u}\|_2)^{1+\varrho}
$$
with $\varrho\in(0,1]$.
\end{theorem}
\begin{proof}
Firstly,  the $\nabla \Theta(u)$ has been known to be strongly semismooth because the proximal mapping of the $\ell_p$-norm function with $p=1,2, \infty$ is strongly semismooth.
Secondly, note that $\hat{u}:= \tau^{-1}\bar{u}-b+\tilde{b}^k$,
then (\ref{ybar}) can be rewritten as
$\bar{y}=\text{Prox}_{\tau^{-1} \| \cdot \|_{p}}(\hat{u})$, or equivalently, from the Moreau's identity (\ref{moreau}) or Lemma \ref{lemma11} that
\begin{equation}\label{yy}
 \bar{y}=\hat{u}-\tau^{-1}\Pi_{B_q^{(1)}}(\tau \hat{u}),
\end{equation}
where $q$ satisfies $1/p+1/q=1$ and operations $1/\infty=0$ and $1/0=\infty$ are followed.
The Assumption \ref{assum32} indicates $\bar{y}\neq 0$ by noting the notation $\bar y=A\bar x-b$ defined in (\ref{PMMmodel2}), which means
$$
\hat{u}\neq \tau^{-1}(\Pi_{B_{q}^{(1)}}\tau \hat{u}).
$$
At the case of $p=1$, from Lemma \ref{lemma11} (a) and Assumption \ref{assum32}, we know for each $i$ that $|\hat{u}_i|>1/\tau$, and hence $\partial \text{Prox}_{\tau^{-1} \| \cdot \|_{1}}(\hat{u})$
is positive definite from Lemma \ref{lemma22} (a).
At the case of $p=2$, the fact $A\bar x-b\neq 0$ and Lemma \ref{lemma11} (b) shows that $\|\hat{u}\|_2>1/\tau$, which further indicates that
$\partial \text{Prox}_{\tau^{-1} \| \cdot \|_{2}}(\hat{u})$ is positive definite from Lemma \ref{lemma22} (b).
At the case of $p=\infty$,  the fact $A\bar x-b\neq 0$ or equivalently $\hat{u}\neq \tau^{-1}\Pi_{B_{q}^{(1)}}(\tau \hat{u})$,
means that $\|\hat{u}\|_1>1/\tau$ from Lemma \ref{lemma11} (c). Therefore,
$\partial \text{Prox}_{\tau^{-1} \| \cdot \|_{\infty}}(\hat{u})$ is positive definite from Lemma \ref{lemma22} (c) and Assumption \ref{assum33}.
Taking everything together, these analysis proved that the generalized Jocobian
$\tilde{\partial}^2\Theta(\bar{u})$ given in (\ref{partt}) is positive definite.
Based on this fact, the convergence rate of $\{u^j\}$ can be directly obtained from \cite[Theorem 3.5]{li38}.
\end{proof}

\section{Numerical experiments}\label{num4}

In this section, we conduct some numerical experiments to demonstrate the superiority of model (\ref{lp}) and the practical performance of Algorithm PMM\_$\ell_p$-$\ell_{1-2}$.
All the experiments are performed with Microsoft Windows 10 and MATLAB R2018a, and run on a PC with an Intel Core i7 CPU at 1.80 GHz and 8 GB of memory.

\subsection{General descriptions}

We conduct experiments on two types of sensing matrices. Define the test sensing matrix as $A=[a_1,\ldots, a_n]\in \mathbb{R}^{m\times n}$. One type matrix of incoherent with high probability is the random Gaussian matrix defined as
$$
a_{i} \stackrel{\text { i.i.d. } }{\sim} \mathcal{N}(0, I_{m} / m), \quad i=1, \ldots, n,
$$
and the random partial DCT matrix with the following expression
$$
a_{i}=\frac{1}{\sqrt{m}} \cos (2 i \pi \xi), \quad i=1, \ldots, n,
$$
where $\xi\in \mathbb{R}^m\sim\mathcal{U}([0,1]^m)$, i.e., the components of $\xi$ are uniformly and independently sampled from $[0,1]$.
The other type is ill-conditioned  matrix of significantly  higher coherence which is  a randomly oversampled partial DCT  matrix
$$
a_{i}=\frac{1}{\sqrt{m}} \cos (2 i \pi \xi / t), \quad i=1, \ldots, n,
$$
where $\xi\in \mathbb{R}^m\sim\mathcal{U}([0,1]^m)$ and $t$ is the parameter used to decide how coherent the matrix is. The larger $t$ is, then the higher the coherence is.

For comparison in a relatively fair way, we measure the quality of the reconstruction solutions using the relative error defined as
$$
  \text{RLNE}:=\frac{\|\bar x-\underline{x}\|_2}{\|\underline{x}\|_2},
$$
where $\bar x$ and $\underline{x}$ are the reconstructed and ground truth signals, respectively.
Besides, it has been shown in Theorem \ref{the2} (c) that $\|x^{k+1}-x^k\|_2\rightarrow 0$ as $k\rightarrow\infty$, so it is reasonable to terminate the iteration if
$$
\text{RelErr}:=\frac{\|x^{k+1}-x^k\|_2}{\max\{\|x^k\|_2,1\}}
$$
is sufficiently small. In each tested algorithm, we stop the iterative process if RelErr$\leq 1e-6$ or the iteration number achieves $2000$.
We fix the parameters $\rho=0.999$, $\mu=0.1$ and $\delta=0.5$, other parameters' values will  be determined adaptively at each experiment.

\subsection{Verify the superiorities of model (\ref{lp})}
In this part, to numerically show the superiorities of model (\ref{lp}), we also test against (\ref{l1l2}) by the using of two state-of-the-art solvers  ADMM\_$\ell_{1-2}$ and FB\_$\ell_{1-2}$ of \cite{lou15}.
The Matlab
package of both solvers is available at \url{https://github.com/mingyan08/ProxL1-L2} where the parameters are left to be its default settings.
\subsubsection{Test the robustness of $\|Ax-b\|_p$ under different noise types}
\begin{figure}[htbp]
\centering
\begin{subfigure}[b]{0.32\textwidth}
\includegraphics[width=\textwidth]{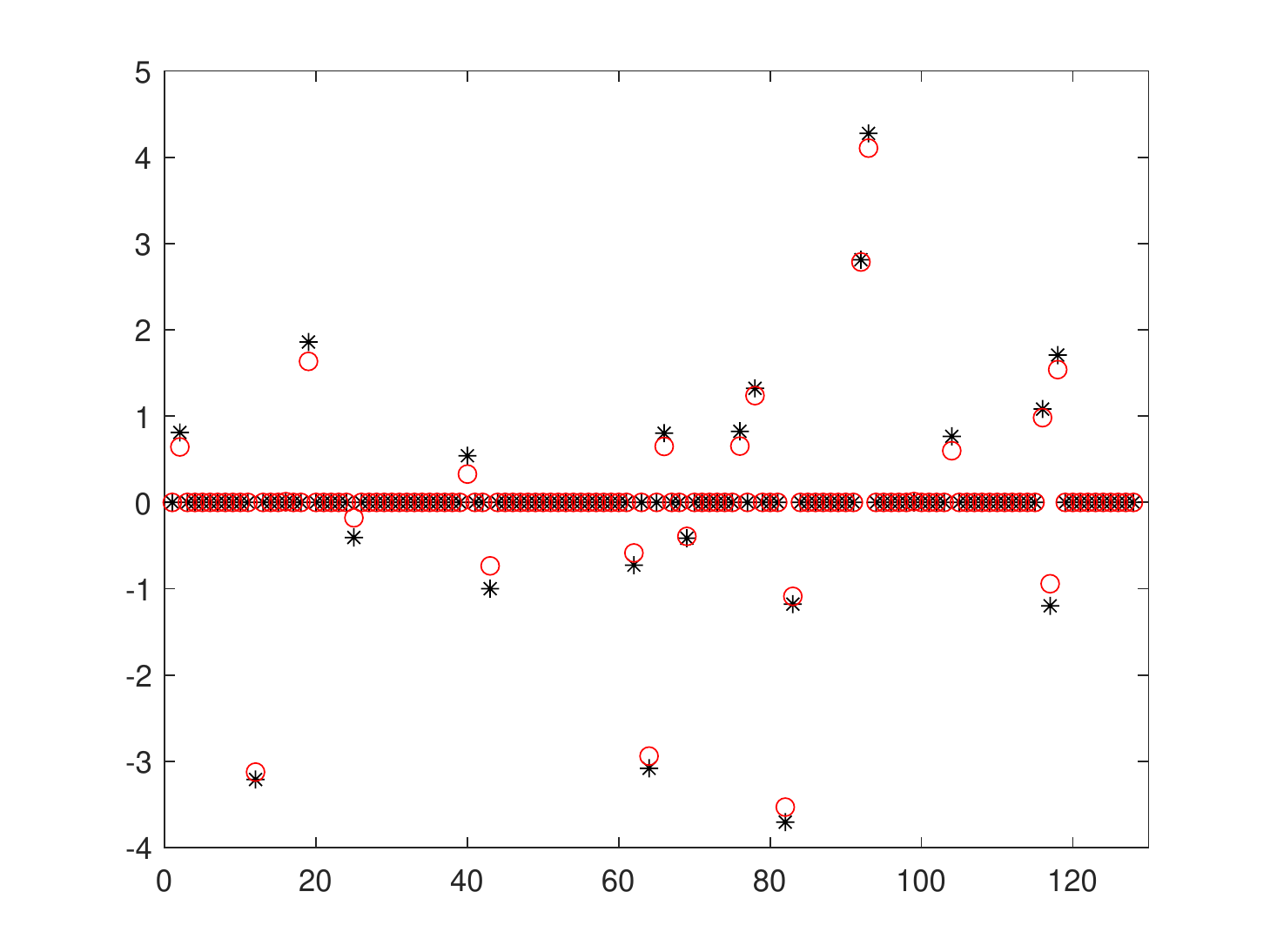}
\caption{RLNE $=6.17e-2$}
\end{subfigure}
\begin{subfigure}[b]{0.32\textwidth}
\includegraphics[width=\textwidth]{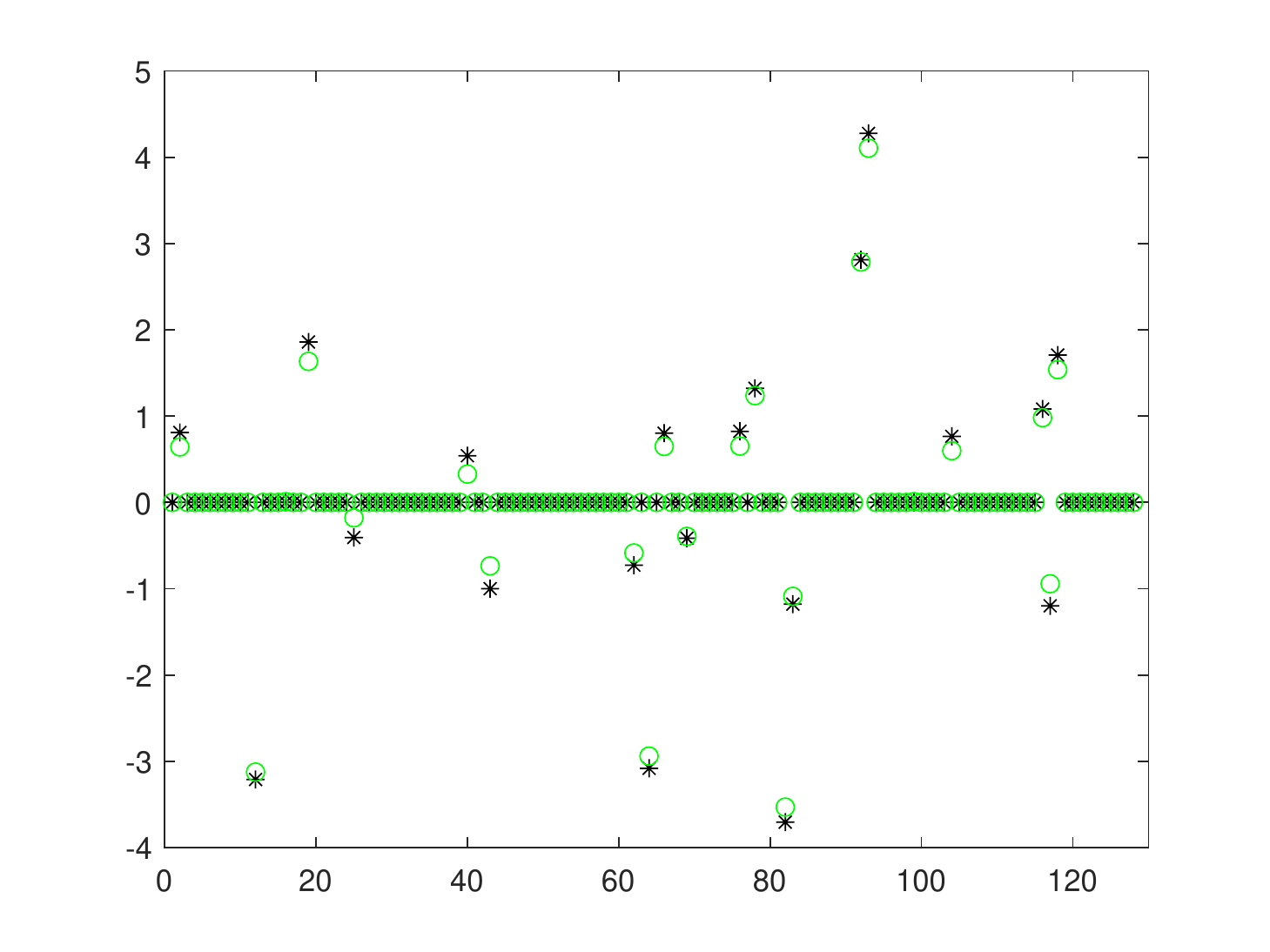}
\caption{RLNE $=6.17e-2$ }
\end{subfigure}
\begin{subfigure}[b]{0.32\textwidth}
\includegraphics[width=\textwidth]{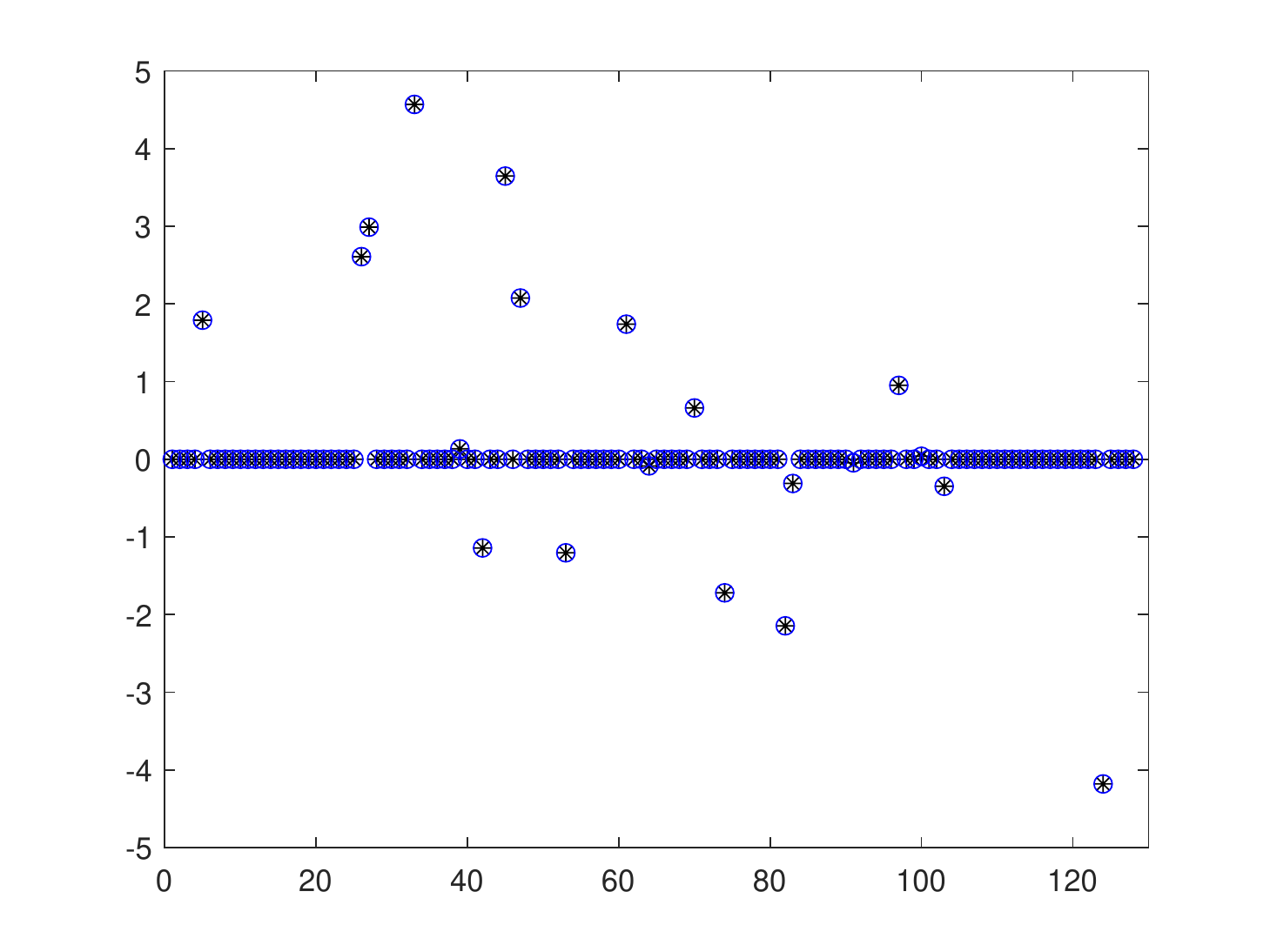}
\caption{RLNE $=8.47e-8$}
\end{subfigure}
\vspace{-.3cm}
\caption{\itshape Log-normal noise: the original signal (black stars) versus the recovery signals by ADMM\_$\ell_{1-2}$ (red circles),  FB\_$\ell_{1-2}$ (green circles), and by PMM\_$\ell_p$-$\ell_{1-2}$ (blue circles).}
\label{fig1}
\vspace{-.3cm}
\end{figure}
\begin{figure}[htbp]
\centering
\begin{subfigure}[b]{0.32\textwidth}
\includegraphics[width=\textwidth]{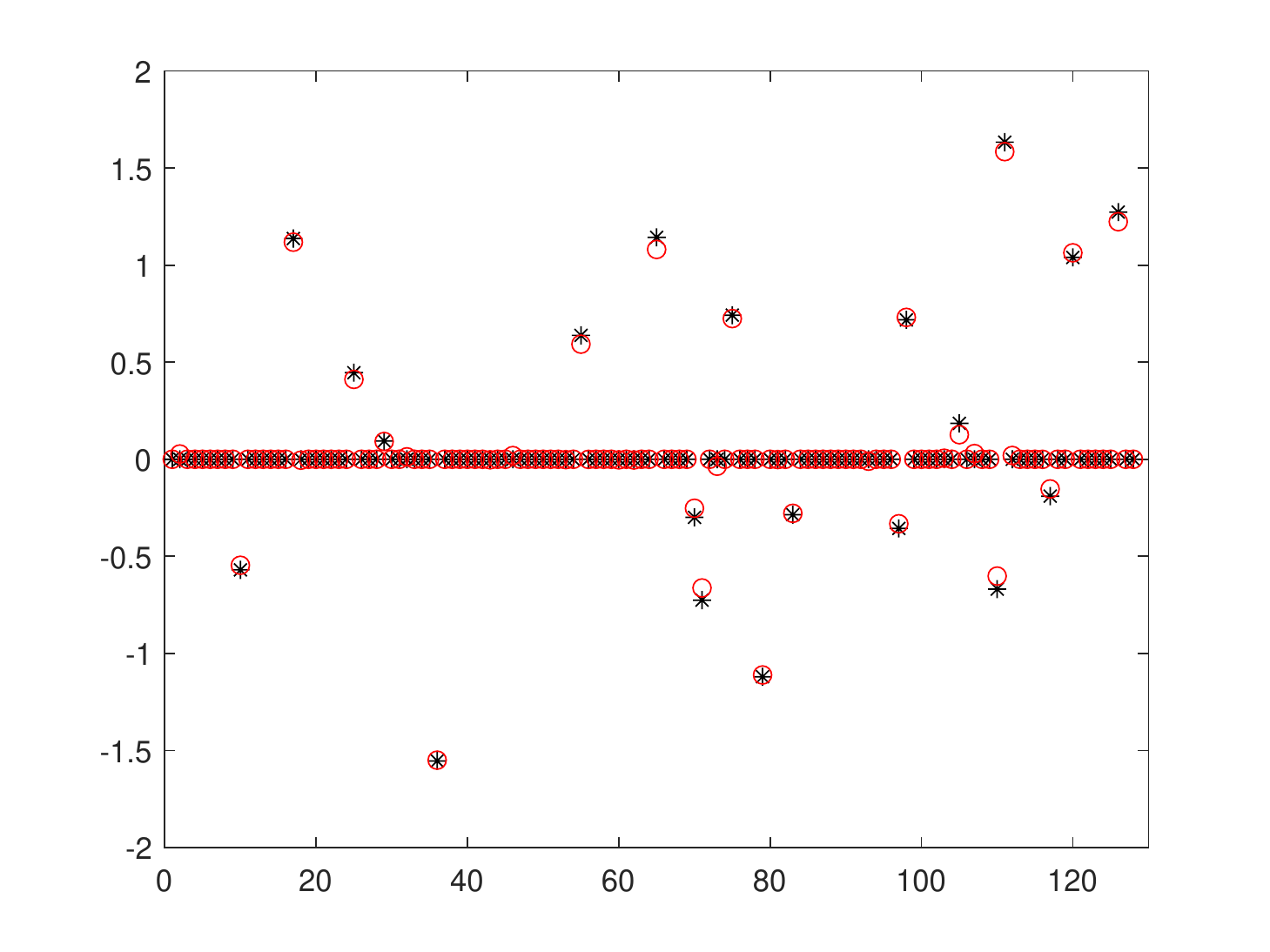}
\caption{RLNE $=4.76e-2$}
\end{subfigure}
\begin{subfigure}[b]{0.32\textwidth}
\includegraphics[width=\textwidth]{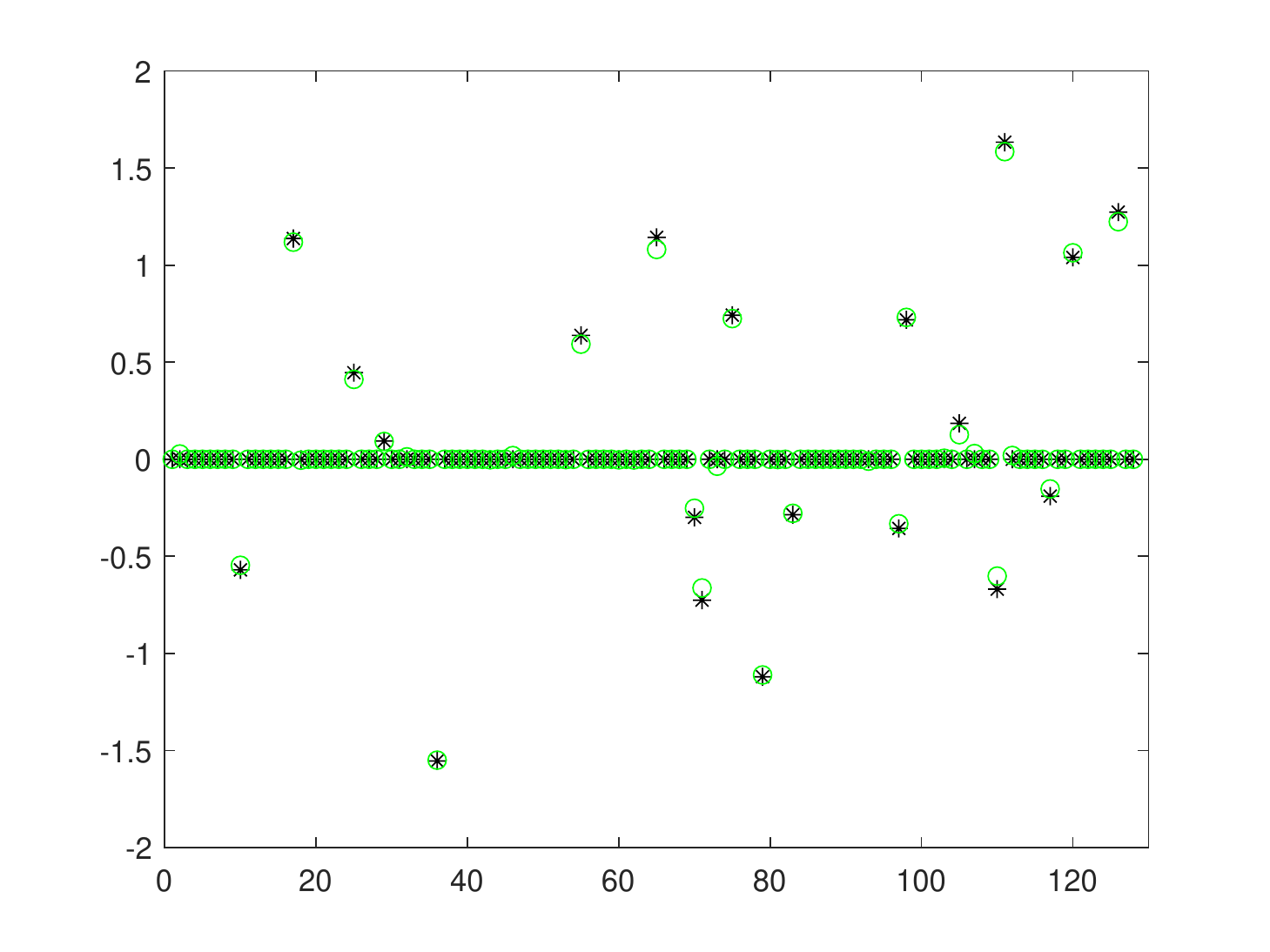}
\caption{RLNE $=4.76e-2$ }
\end{subfigure}
\begin{subfigure}[b]{0.32\textwidth}
\includegraphics[width=\textwidth]{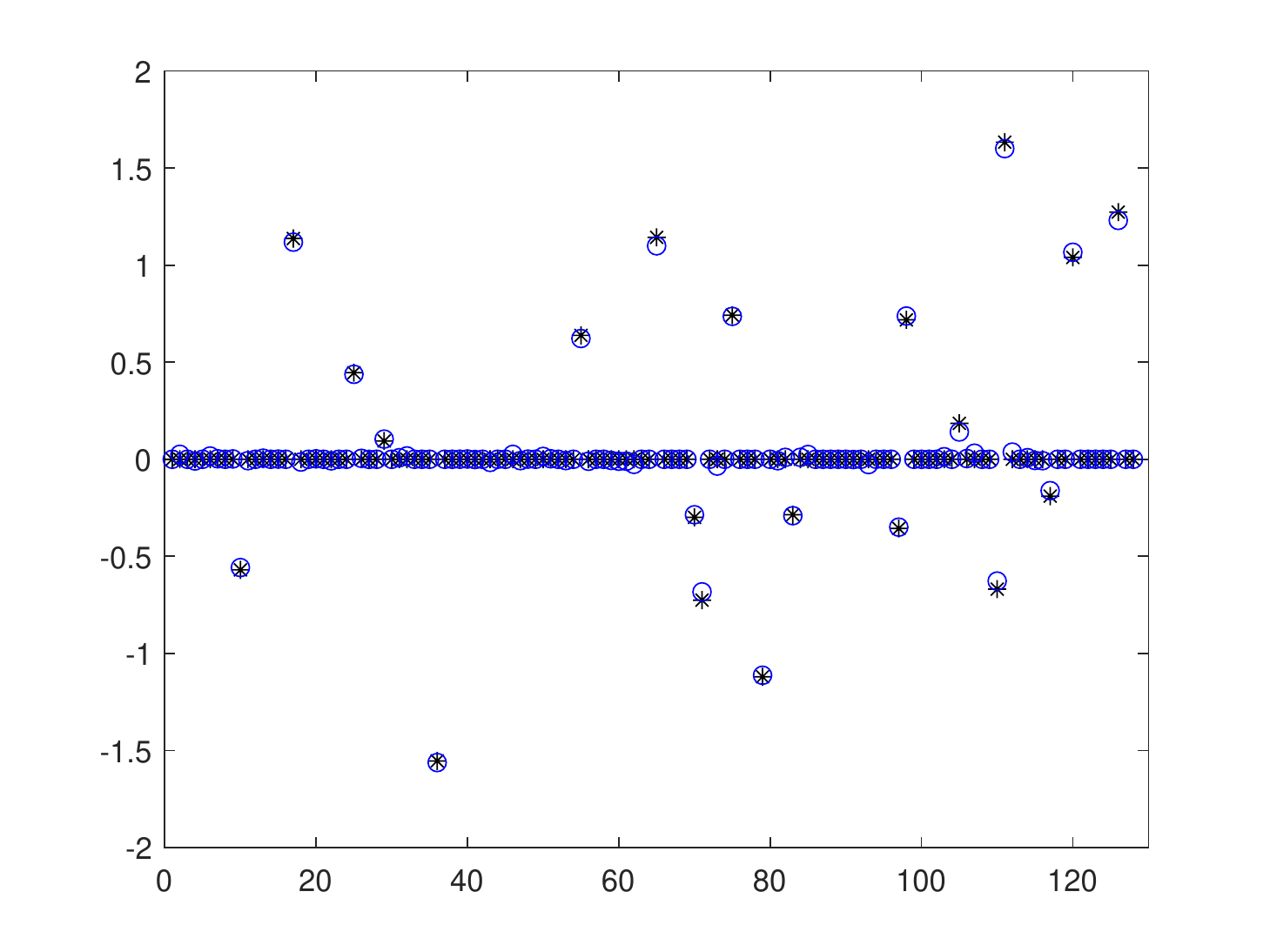}
\caption{RLNE $=3.92e-2$}
\end{subfigure}
\vspace{-.3cm}
\caption{\itshape Gaussian noise: the original signal (black stars) versus the recovery signals by ADMM\_$\ell_{1-2}$ (red circles),  FB\_$\ell_{1-2}$ (green circles), and by PMM\_$\ell_p$-$\ell_{1-2}$ (blue circles).}
\label{fig2}
\vspace{-.3cm}
\end{figure}
\begin{figure}[htbp]
\centering
\begin{subfigure}[b]{0.32\textwidth}
\includegraphics[width=\textwidth]{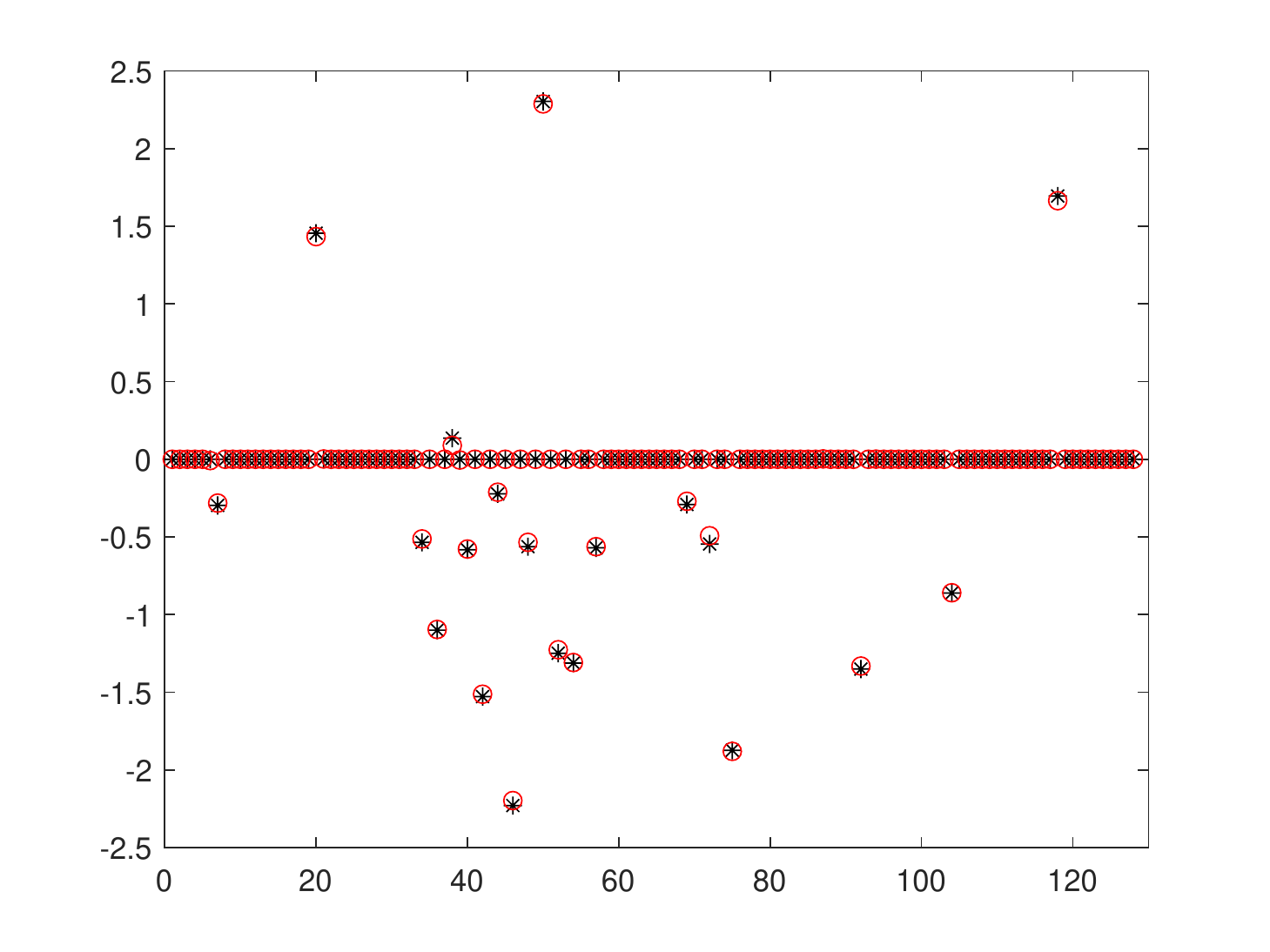}
\caption{RLNE $=1.92e-2$}
\end{subfigure}
\begin{subfigure}[b]{0.32\textwidth}
\includegraphics[width=\textwidth]{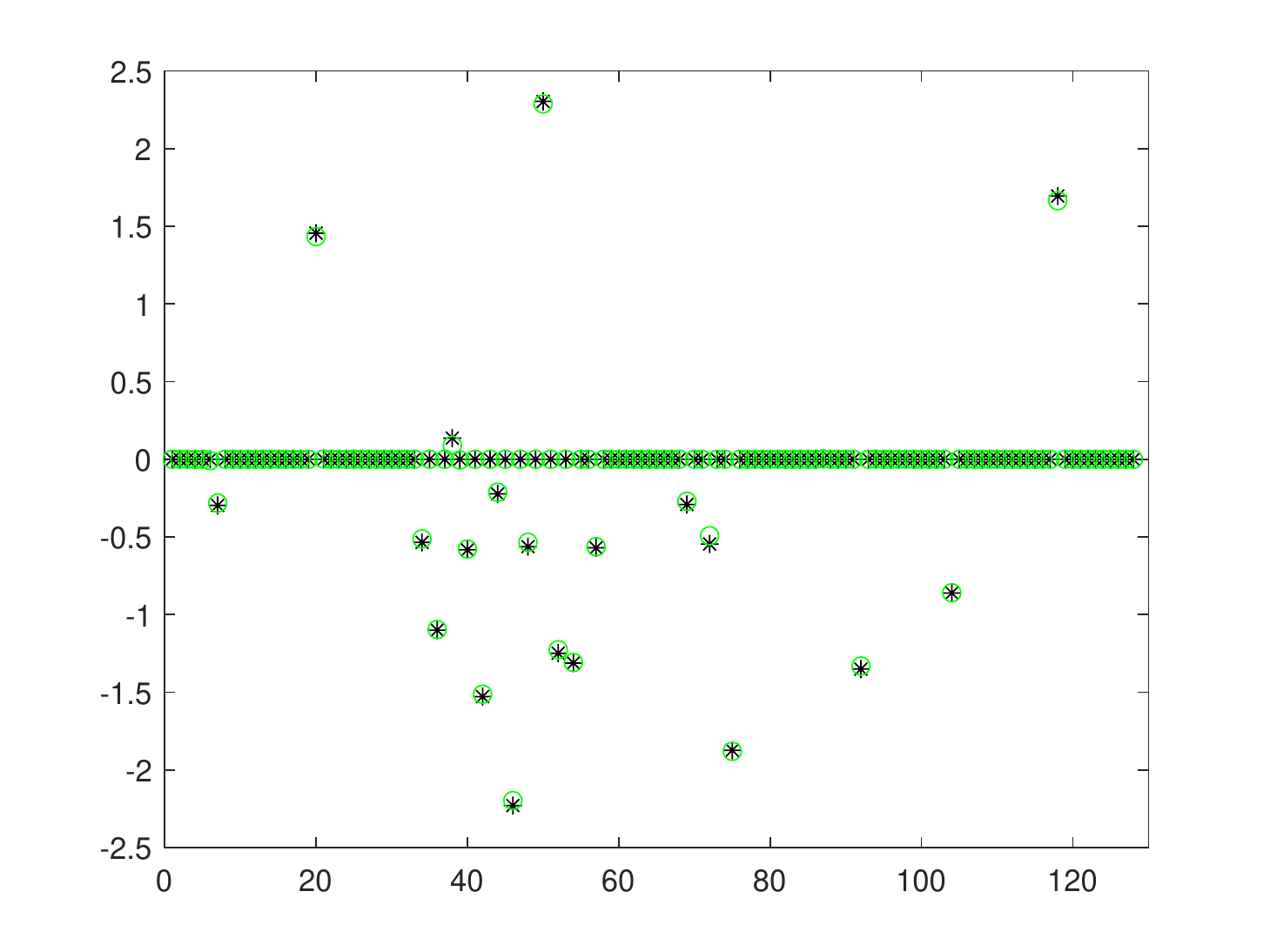}
\caption{RLNE $=1.92e-2$ }
\end{subfigure}
\begin{subfigure}[b]{0.32\textwidth}
\includegraphics[width=\textwidth]{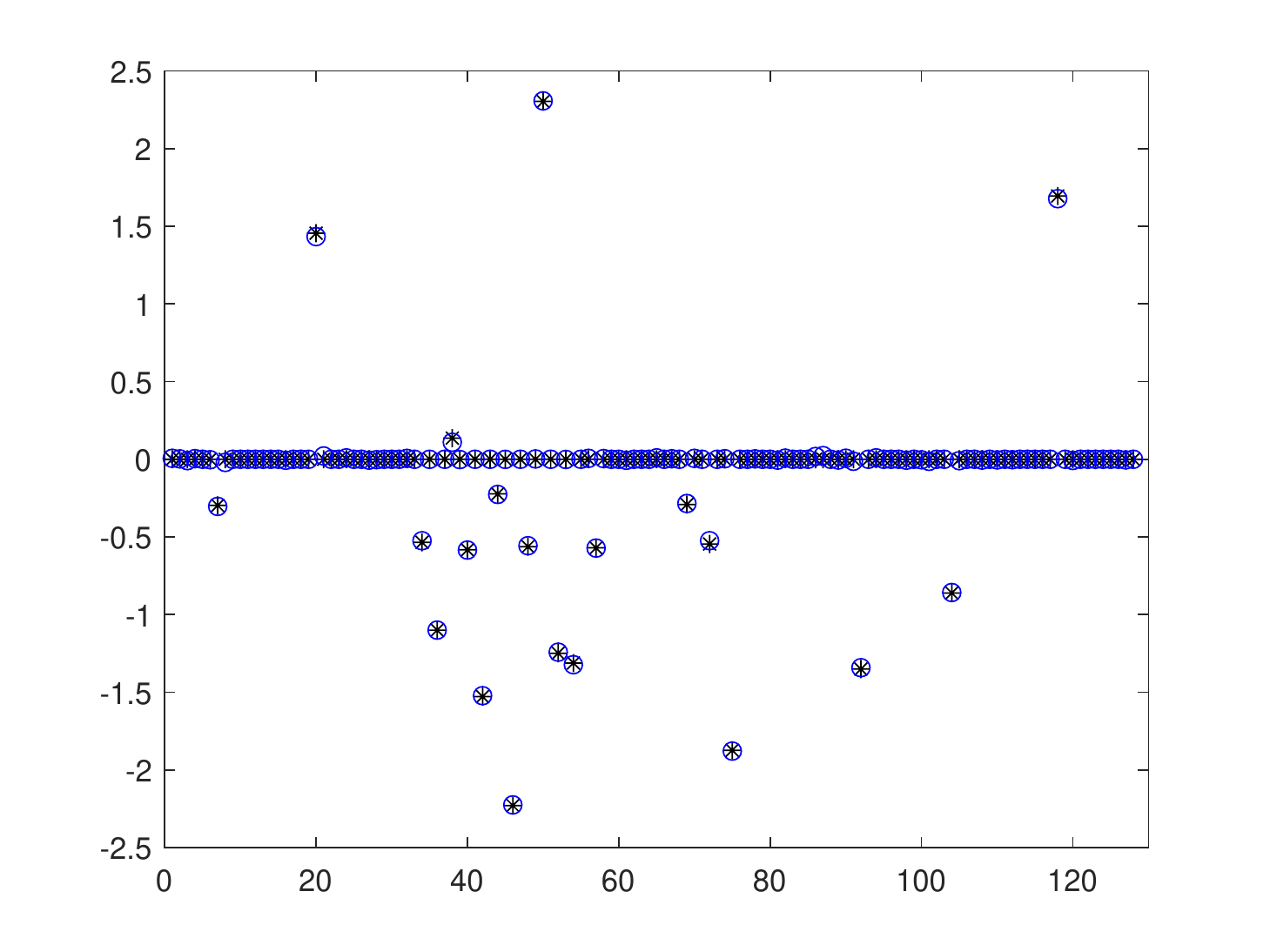}
\caption{RLNE $=1.37e-2$}
\end{subfigure}
\vspace{-.3cm}
\caption{\itshape Uniform noise: the original signal (black stars) versus the recovery signals by ADMM\_$\ell_{1-2}$ (red circles),  FB\_$\ell_{1-2}$ (green circles), and by PMM\_$\ell_p$-$\ell_{1-2}$ (blue circles).}
\label{fig3}
\vspace{-.3cm}
\end{figure}
In order to highlight the robustness and practicality of the model (\ref{lp}), we mainly test three different types of noises in this part. The observation $b$ is obtained by
$$
b=A*\underline{x}+\alpha*\text{noise},
$$
where $\alpha$ is the noise level, and ``noise" is the one of the following types of noise: log-normal noise, Gaussian noise, and uniform noise.
Here, we recover a sparse signal of being polluted by different kinds of noises via model (\ref{lp}) with different $p$ values.
In this test, we set the noise level as $\alpha=1e-2$ and choose a random partial DCT to be a sensing matrix with size $64\times128$.
We say that the a signal $x$ is $K$-sparse if the number of nonzero entries in $x$ is $K$.
In this test, we set $K=20$ to the ground truth signal $\underline{x}$. Finally, we also set $\beta\equiv1$ for the sake of simplicity.

At the case of log-normal noise, we set $p=1$, i.e., the data fidelity term is $\|Ax-b\|_1$. Besides, the weighting parameter $\lambda$ in (\ref{lp}) is chosen as $\lambda=8e-2$.
In Algorithm PMM\_$\ell_p$-$\ell_{1-2}$, we let $\tau^0=1e-1$ and $\sigma^0=\sqrt{2}\|AA^{\top}\|_2$, and in algorithms ADMM\_$\ell_{1-2}$ and FB\_$\ell_{1-2}$,
all the parameters are left to be their default settings. The original signal and the reconstructed signals recovered by
PMM\_$\ell_p$-$\ell_{1-2}$, ADMM\_$\ell_{1-2}$, and FB\_$\ell_{1-2}$ are listed respectively in Figure \ref{fig1}. In this figure, the original signal is denoted by black stars ``$\ast$" and the recovered signals are denoted by ``$\bigcirc$" marked in different color.
Comparing each plot from left to right,  we clearly see that all the stars in (c) are circled exactly by the blue circles with a symbol ``$\circledast$" which indicates that the using of
$\|Ax-b\|_1$ is  better.
Moreover, we also see that the final relative error of the solution derived by PMM\_$\ell_p$-$\ell_{1-2}$ is $8.47e-8$ which is significantly smaller than
the one produced by other algorithm, which once again indicates that
the advantage of $\|Ax-b\|_1$.
At the case of Gaussian noise, we set $p=2$ and  $\tau^0=2$ , i.e., the data fidelity term is $\|Ax-b\|_2$, and at the case of uniform noise, we set $p=\infty$ and  $\tau^0=1e-2$, i.e., the data fidelity term is $\|Ax-b\|_\infty$.
 Other parameters' values in both cases are chosen as
$\lambda=1e-2$ and $\sigma^0=\sqrt{2}\|AA^{\top}\|_2$.
To be fair, we find that the two comparison algorithms perform better when $\lambda=1e-2$, so values of the
parameters $\lambda=1e-2$ are fixed in ADMM\_$\ell_{1-2}$ and FB\_$\ell_{1-2}$ for $p=2, \infty$.
The results of each algorithm for both different types of noise are listed  in Figures \ref{fig2} and \ref{fig3}, respectively.
From these figures, we can visibly see that the quality of the solution derived by PMM\_$\ell_p$-$\ell_{1-2}$ is better.
From these limited numerical experiments, it can be concluded that, as far as the three types of noise are concerned, our proposed
model (\ref{lp}) has the ability to get produce higher quality reconstruction results if the data fidelity term is chosen adaptively.
\subsubsection{Test the superiority of $\ell_{1-2}$-term to $\ell_1$-norm under different sensing matrices and sparsity levels}

In this section, we test the superiority of the $\ell_{1-2}$-term in two ways, i.e., using different sensing matrices and using different sparsity levels.
To address the first issue, we test PMM\_$\ell_p$-$\ell_{1-2}$, ADMM\_$\ell_{1-2}$ and FB\_$\ell_{1-2}$ repeatedly by the using of three types of sensing matrix, say
random Gaussian matrix (GAUS), random partial DCT matrix (PDCT), and randomly oversampled partial DCT (ODCT).
At each tested case, we run each algorithm based on two types of sparsity. All the parameters' values are taken as the same as the ones previously except for the noise level $\alpha=1e-3$ and the weighting parameter $\beta$. In this test, the case of $\beta\equiv0$ means that only a $\ell_1$-norm is used. The results of each algorithm with respect to the final relative error are listed in Table \ref{tab1}.
From Table \ref{tab1}, we see that the RLNE values at the last column are always smaller than the corresponding ones at other two columns,
which once again shows that our proposed model (\ref{lp}) indeed benefits the quality of the reconstruction solutions.
Observing the results row-by-row, we find that, at most cases, the values at the case of $\beta=1$ are relatively smaller, which indicates that the ``$-\|\cdot\|_2$"-term  has the potential ability to extract sparse property.

\begin{table}[h]
\centering {\scriptsize\caption{Final RLNE values of each algorithm}
\begin{tabular}{cccc|c|c | c }
\hline
\multicolumn{4}{c|}{} & \multicolumn{1}{c|}{} & \multicolumn{1}{c|}{} & \multicolumn{1}{c}{}\\

Sensing matrix&Dimension&Sparsity&$\beta$&ADMM\_$\ell_{1-2}$&FB\_$\ell_{1-2}$&PMM\_$\ell_p$-$\ell_{1-2}$\\
\hline
\multirow{4}{0.8cm}{GAUS}
&$50\times 100 $&5 &0&2.03e-2 &2.03e-2  &4.80e-3  \\
&$50\times 100 $&5 &1&2.01e-2 &2.01e-2  &{5.30e-3}\\
&$400\times 800$&20&0&2.17e-2 &2.17e-2   &6.40e-3\\
&$400\times 800$&20&1&2.04e-2 &2.04e-2 &{6.10e-3}\\
\hline
\multirow{4}{0.8cm}{PDCT}
&$50\times 100$ &5&0&3.54e-2 &3.54e-2  &8.00e-3\\
&$50\times 100$&5&1&1.56e-2&1.56e-2&{5.40e-3}\\
&$400\times 800$&20&0&2.19e-2 &2.19e-2  &6.30e-3\\
&$400\times 800$&20&1&1.94e-2 &1.94e-2  &{6.20e-3}\\
\hline
\multirow{4}{0.8cm}{ODCT (t=10)}
&$100\times 200$ &5&0&5.22e-1&4.48e-1&4.39e-2\\
&$100\times 200$ &5&1&1.08e-1&1.08e-1 &{1.64e-2}\\
&$400\times 800$&10&0&7.40e-1 &7.41e-1 &3.08e-2\\
&$400\times 800$&10&1&3.74e-1 &3.07e-1 &{1.51e-2}\\
\hline
\multirow{4}{0.8cm}{ODCT (t=15)}
&$100\times 200$  &5&0&6.77e-1 &6.08e-1 &2.44e-1\\
&$100\times 200$  &5&1&1.09e-1&1.09e-1&{3.08e-2}\\
&$400\times 800$ &10&0&2.94e-1 &2.30e-1  &3.78e-1\\
&$400\times 800$ &10&1&1.36e-1 &1.36e-1  &{9.70e-3}\\
\hline
\end{tabular}\label{tab1}
}
\end{table}

We now turn our attention to using different sparsity levels to test the superiority of the $\ell_{1-2}$-term.
In this test, the sensing matrices are chosen as the random Gaussian matrix with size $200\times500$ and the randomly oversampled partial DCT matrix with $t=15$.
For the sake of simplicity, we fix $\beta\equiv1$ in the $\|x\|_1-\beta\|x\|_2$. Moreover, we also choose $\beta\equiv0$ for comparison and the names of the corresponding algorithms are abbreviated as ``$***\_{\ell_1}$".
The original signal $\underline{x}$  tested on random Gaussian matrix has a support size of $5 : 2 : 25$, which means that the support size starts from $5$ and ends at $25$
with an increment of $2$.
The true signal $\underline{x}$  used at the randomly oversampled partial DCT matrix case has a support size of $5:1:15$.
The statistical results on each tested case drawn in Figure \ref{fig4}.
 \begin{figure}[htbp]
\centering
\begin{subfigure}[b]{0.45\textwidth}
\includegraphics[width=\textwidth]{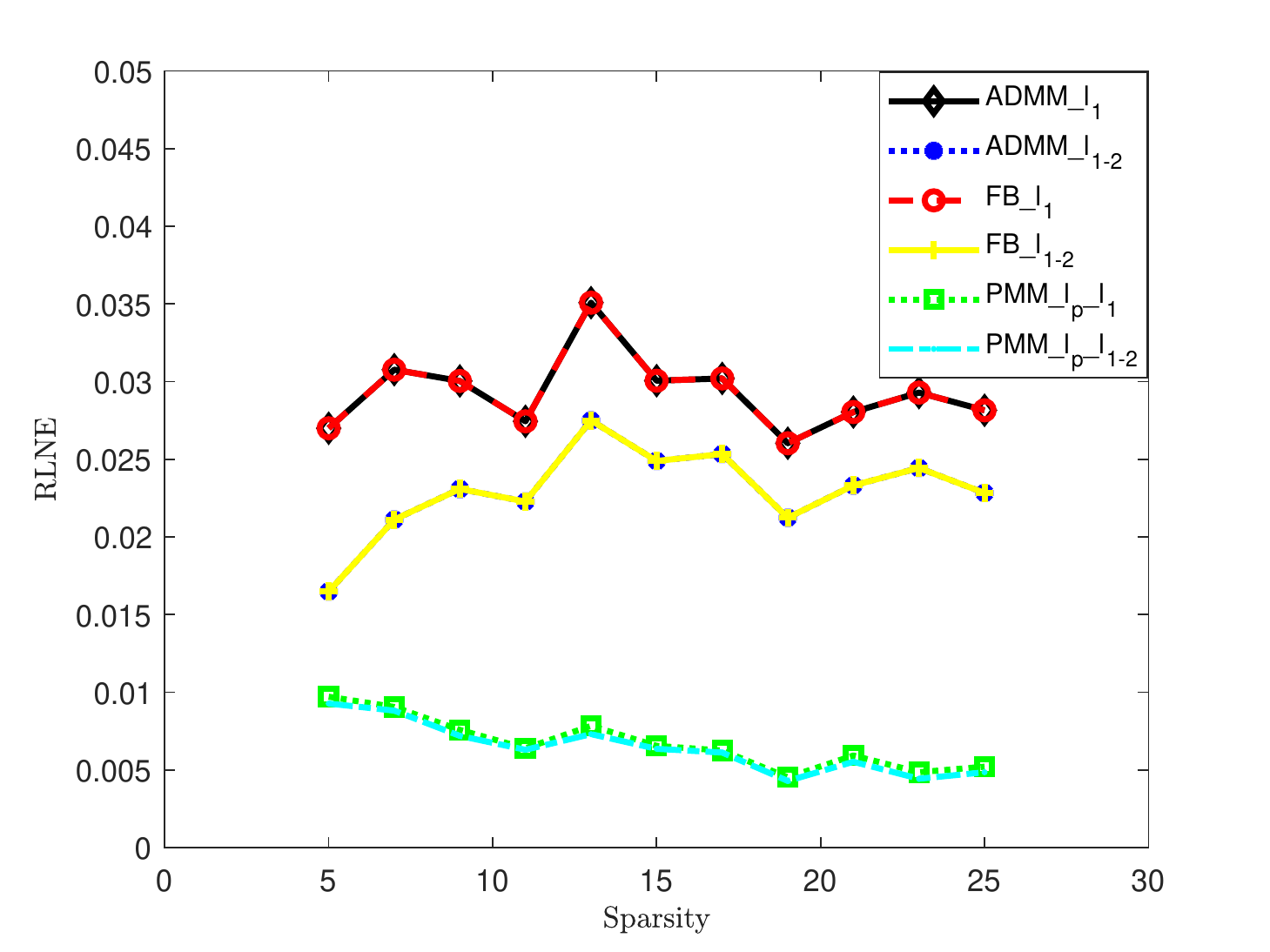}
\caption{Gaussian matrix}
\end{subfigure}
\begin{subfigure}[b]{0.45\textwidth}
\includegraphics[width=\textwidth]{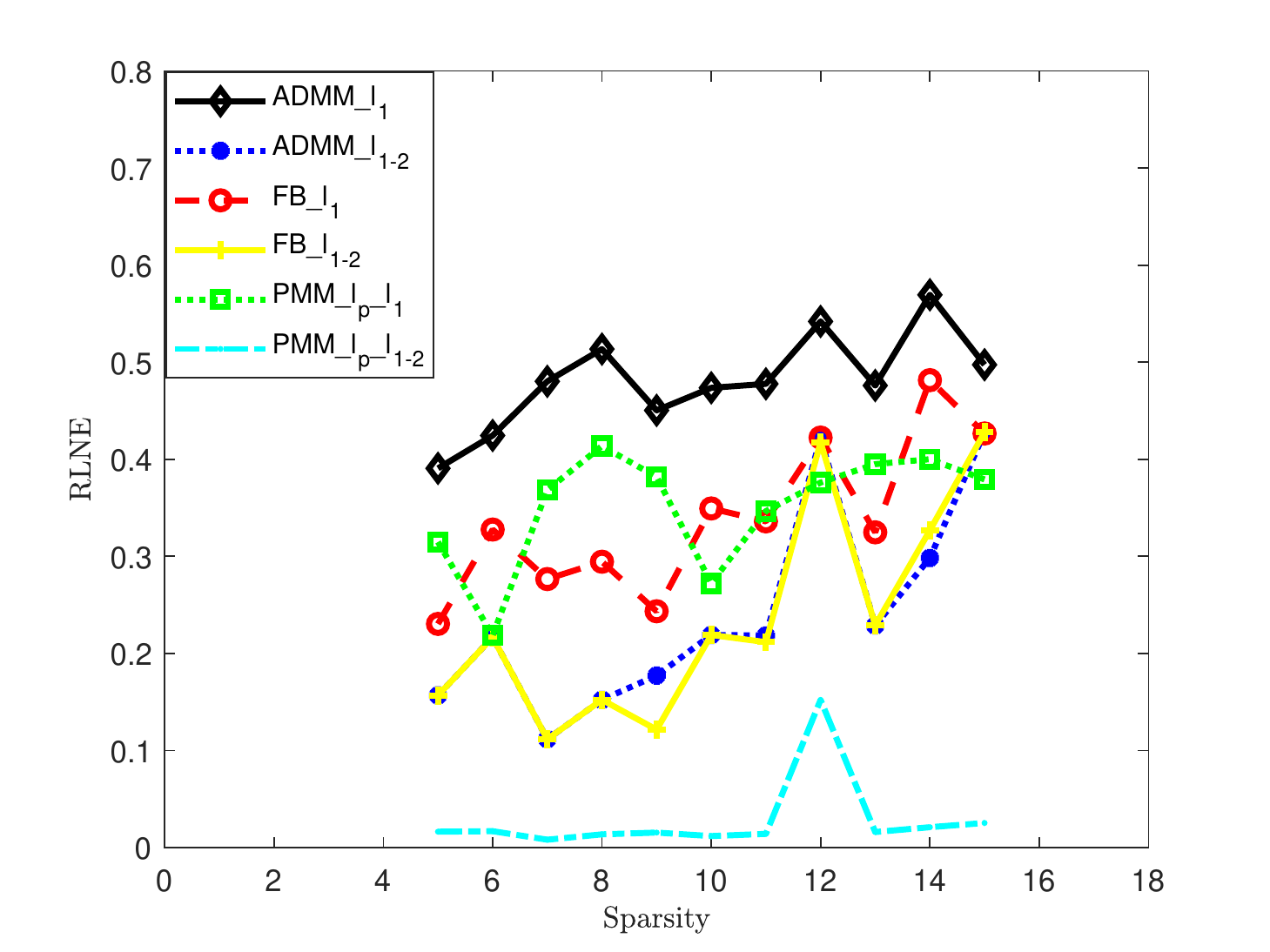}
\caption{ODCT matrix with $t=15$}
\end{subfigure}
\vspace{-.3cm}
\caption{\itshape RLNE values under different sparsity
 levels}
\label{fig4}
\vspace{-.3cm}
\end{figure}

As can be seen from left plot that, the curves derived by PMM\_$\ell_p$-$\ell_{1}$ and PMM\_$\ell_p$-$\ell_{1-2}$ are always at the bottom which show that our proposed approach is a winner.
Moreover, we also see that the curves derived by ``$***\_{\ell_1}$" and ``$***\_{\ell_{1-2}}$" are almost the same, which indicates that the ``$-\|\cdot\|_2$" term has hardly any influence
at the random Gaussian matrix case.
While we turn our attention to the right plot, it is easy to see that the curves derived by PMM\_$\ell_p$-$\ell_{1}$ lie under the one by ADMM\_$\ell_{1-2}$ and FB\_$\ell_{1-2}$ which once again shows our proposed approach is better. However, at each test case,  the curves derived by ``$***\_{\ell_1}$" and ``$***\_{\ell_{1-2}}$" are totally different, and the curves based on the ``$-\|\cdot\|_2$" term are always at the bottom. From this phenomenon, we conclude that the ``$-\|\cdot\|_2$" term is capable of improving the quality of the reconstruction solutions. Taking everything together, these experiments  under different sensing matrices and different sparsity levels show that the superiority of the proposed model in recovering sparse signals is obvious.

\subsection{Evaluate the Performance of PMM\_$\ell_p$-$\ell_{1-2}$ }
In this part, we test PMM\_$\ell_p$-$\ell_{1-2}$ against DCA\_ADMM --- a the difference of convex functions algorithm (DCA) in which the subproblem is solved by
alternating direction method of multipliers.
We firstly describe some implementation details of using ADMM to solve the second subproblem in (\ref{dc}).
Denote $y:=Ax-b$. The subproblem of DCA exhibited in (\ref{dc}) can be written equivalently as
\begin{equation}\label{p1}
\begin{array}{lll}
\min\limits_{x,y} & \|y\|_p +\lambda\|x\|_{1}-\lambda\beta\langle v^k,x\rangle  \\
\text{s.t.} & Ax-y=b.\\
\end{array}
\end{equation}
Let $\sigma>0$ be a penalty parameter, the augmented Lagrangian function associated with problem (\ref{p1})  is given by
\begin{align*}
\mathcal{L}_{\sigma}(x,y;u)&= \|y\|_p +\lambda\|x\|_{1}-\lambda\beta\langle v^k,x\rangle+\langle u, Ax-y-b\rangle+\frac{\sigma}{2}\|Ax-y-b\|^2_2,
\end{align*}
where $u\in\mathbb{R}^{m}$ is a  multiplier associated with the constraint.
Staring from an initial point $(x^{0},y^{0};u^{0})$, the semi-proximal ADMM for solving (\ref{p1}) is summarized as
\begin{equation}\label{sadmm}
\left\{
\begin{array}{l}
y^{j+1}=\argmin_{y}\big\{\|y\|_p+\frac{\sigma}{2}
\|Ax^j-y-b+u^j/\sigma\|_2^2+\frac{\sigma}{2}\|y-y^j\|_{S}^2\big\},
\\[3mm]
x^{j+1}=\argmin_{x}\big\{\lambda\|x\|_{1}-\lambda\beta\langle v^k,x\rangle+\frac{\sigma}{2}
\|Ax-y^{j+1}-b+u^j/\sigma\|_2^2+\frac{\sigma}{2}\|x-x^j\|_{T}^2\big\},
\\[3mm]
u^{j+1}=u^{j}+\tau\sigma\big(Ax^{j+1}-y^{j+1}-b\big),
\end{array}
\right.
\end{equation}
where $S$ and $T$ are weighted positive semi-definite matrices, and $\tau$ is steplength chosen in the interval $(0,(1+\sqrt{5})/2)$.
Choose $S:=0$ and $T:=\zeta I-\A^\top A$ where $\zeta>0$ be a positive scalar such that $T$ be positive semi-definite.
It is a trivial task to deduce that the $y$- and $x$-subproblems can be written as the following proximal mapping forms
$$
y^{j+1}=\operatorname{Prox}_{{\sigma}^{-1}\|\cdot\|_p}\Big(Ax^j-b+u^j/\sigma\Big),
$$
and
$$
x^{j+1}=\operatorname{Prox}_{(\zeta\sigma)^{-1}\lambda \|\cdot\|_1}\Big(\frac{A^\top[y^{j+1}+b-u^j/\sigma]+[\zeta I-A^\top A]x^j}{\zeta}+\frac{\lambda\beta v^k}{\zeta\sigma}\Big),
$$
which means that the iterative scheme (\ref{sadmm}) is easily implementable in the sense that both subproblems admit explicit form solutions from Lemma \ref{lemma11}.
We must emphasize that the iterative framework (\ref{sadmm}) is actually an implementation of the semi-proximal ADMM of Fazel et al. \cite{SEMP13}.
Hence, its convergence can be guaranteed under some technical conditions. For more details, one may refer to \cite[Theorem B]{SEMP13}.

We now compare the numerical performance of two different methods for solving problem (\ref{lp}).
In this test, we set the noise level as $\alpha=1e-3$ and choose diverse sensing matrices and different sparsity.
For Algorithm PMM\_$\ell_p$-$\ell_{1-2}$, we set the same parameters as Section 4.2.
For fairness, we take the same penalty parameter $\lambda$ and set $\beta=1$ in model (\ref{lp}) for both algorithms.
In addition, the parameter ${\sigma}$ in DCA\_ADMM and the initial $\sigma^0$ in PMM\_$\ell_p$-$\ell_{1-2}$ are fixed as $\sqrt{2}\|AA^{\top}\|_2$.
These settings always make both algorithms work well during the experiments' preparations.
For the $\ell_p$-norm in data fidelity term, we choose $p=1$ at the case of log-normal noise (LN), $p=2$ for Gaussian noise (GN), and $p=\infty$ for uniform noise (UN).
Detailed numerical results are reported in Table \ref{tab41} which contains the names of the noise types (Noise), the type of
sensing matrix (Matrix) with its dimensions (Dim), the CPU time required in seconds (Time), the final
objective function value of (\ref{lp}) (Obj), the values of RLNE and RelErr, and the number of
outer iterations (Iter). Besides, the symbol ``-" presents the algorithm failed to achieve convergence within $2000$ number of iterations.

\begin{table}[htbp]
\centering {\tiny\caption{The performance of the PMM\_$\ell_p$-$\ell_{1-2}$ and DCA\_ADMM.}
\begin{tabular}{lrrrr|rrrr|rrrr}
\hline
\multicolumn{5}{c|}{} & \multicolumn{4}{c|}{PMM\_$\ell_p$-$\ell_{1-2}$} & \multicolumn{4}{c}{DCA\_ADMM}\\
\hline
Noise  & Matrix(Dim) & K &$\lambda$&$\sigma^0(\sigma) $& Time & Obj  & RLNE & Iter & Time & Obj  & RLNE & Iter \\
\hline
\multirow{4}{0.2cm}{LN}
&GAUS
  ($100\times 200$)& 10 & 0.02& 1 &0.0541&0.4873&9.61e-9&15&0.0773&0.4873&3.64e-6&64\\
&GAUS
  ($400\times 800$) & 20 & 0.04& 2 &0.8318&0.8546&1.97e-10&14&1.3320&0.8546&3.65e-6&39\\
&PDCT
  ($200\times 400$) & 10 & 0.06&2 &0.0515&1.3326&4.87e-10&5&0.1005&1.3326&9.38e-7&16\\
&PDCT
  ($400\times 800)$ & 20 &0.08&1.5&0.2232&2.1734&4.37e-10&5&0.5021&2.1734&7.00e-7&14\\
  \hline
\multirow{4}{0.2cm}{GN}
&GAUS
  ($100\times 200$) & 10 &0.005&1&0.1499&0.0281&0.0110&20&0.1612&0.081&0.0110&20\\
&GAUS
  ($400\times 800$) & 20 &0.015&2&4.1701& 0.1795&0.0134&20&6.2627&0.1795&0.0144&16\\
&ODCT(t=5)
  ($100\times 200$) & 10 &0.08&0.1&0.1225& 0.4224&0.0067&8&3.0180&0.4724&0.0040&-\\
&ODCT(t=10)
  ($200\times 400$) & 15 &0.05&0.3&3.0582& 0.5294&0.0179&93&146.3076&0.5385&0.3085&-\\
   \hline
\multirow{4}{0.2cm}{UN}
&GAUS
  ($64\times 128$)& 10 & 0.005&2&2.7091&0.0244&0.0060&213&8.9643&1.7752&1.2280&-\\
&GAUS
  ($128\times 256$) & 15 & 0.001&2&6.1933&0.0084&0.0054&803&22.7704&2.5062&1.2862&-\\
&PDCT
  ($64\times 128$) & 10 & 0.01&2&0.4581&0.0378&0.0028&21&9.0736&2.2083&1.0323&-\\
&PDCT
  ($128\times 256)$ & 15 &0.005&2&2.3322&0.0531&0.0019&51&26.7144&42.7302&0.6863&-\\
\hline
\end{tabular}\label{tab41}
}
\end{table}
At the fist place, we can clearly see that the final RLNE and objective function values are obviously smaller than DCA\_ADMM at all the instances.
On the one hand, we see that PMM\_$\ell_p$-$\ell_{1-2}$ can successfully solve the problem all the instances to the desired accuracy within hundreds or even dozens of steps,
while DCA\_ADMM must be stopped when it reaches the maximum number of $2000$ iterations at some cases.
On the other hand, we can observe that PMM\_$\ell_p$-$\ell_{1-2}$  always takes much less time than DCA\_ADMM.
For example, for the instance GN with ODCT(t=5) matrix, we can see that  PMM\_$\ell_p$-$\ell_{1-2}$ is nearly $30$ times faster than DCA\_ADMM.
In addition, at the ODCT(t=10) sensing matrix case,  PMM\_$\ell_p$-$\ell_{1-2}$ can solve the instance within seconds while
DCA\_ADMM reaches the maximum of $20000$ iterations and consumes more than $2$ minuets but
only produces a rather lower accuracy solution.
Overall, we conclude that PMM\_$\ell_p$-$\ell_{1-2}$ is clearly more robust and efficient than DCA\_ADMM on the limited experiments.

\section{Conclusions}\label{con5}

The compressive sensing theories offered the possibilities to reconstruct a large and sparse signal from highly undersampled data and remove the possible noise simultaneously.
However, the selection of the data fidelity type is known to be much noise depending.
Besides, the efficiencies of almost all the reconstruction models depend heavily on the corresponding numerical algorithms.
Hence, designing a flexible model along with an efficient algorithm which is capable of dealing with more types of noise is especially important.
Using the difference between $\ell_1$-norm and $\ell_2$-norm as a regularization instead of using the $\ell_1$-norm alone has been numerically shown to be more efficiency to extract sparse property even under high coherent condition.
While enjoying these advantages, at the same time, it also bring much difficulties because the ``$\|x\|_1-\beta\|x\|_2$" term is nonsmooth and nonconvex, and the ``$\|Ax-b\|_p$" term is nonsmooth.
To address these issues, we used a proximal majorization technique to make the term ``$-\|x\|_2$" be convex, and then employed a semismooth Newton method to solve the resulting semismooth equations.
We must emphasize that the $x$- and $y$-subproblems in Step $4$ of Algorithm SSN admits analytical solutions if $p$ is chosen as $1$, $2$, or $\infty$, which indicates that the algorithm is easily implementable.
Finally, we did a series of numerical experiments using different noise types, different sensing matrices, and different sparsity levels.
The numerical results showed that the robustness of the proposed model are very evident and the performance of the proposed algorithm is very clear.
Despite this, the Assumption \ref{assum33} to ensure the positive definiteness of the generalized Jacobian $\tilde{\partial}^2\Theta(\bar{u})$ at the case of $p=\infty$ looks strong.
Hence, some theoretical improvements to this issue is highly required.
But this doesn't affect us to believe that PMM\_$\ell_p$-$\ell_{1-2}$ is a valid for sparse signal reconstructing and it may have its own extraordinary potency in other related problems.
At last but not at least, to the best of our knowledge, PMM\_$\ell_p$-$\ell_{1-2}$ is the first algorithm to solve (\ref{lp}), and hence, other efficient algorithms are worthy of designing.
\section*{Acknowledgements}
We would like to thank the anonymous referees for their useful comments and suggestions which
improved this paper greatly.
We would like to thanks professor P. P. Tang from Zhejiang University City College
for her valuable discussions on the proof of Theorem \ref{the33}.
The work of H.  Zhang is supported by the National Natural Science Foundation of
China (Grant No. 11771003). The work of Y. Xiao is supported by the National Natural Science Foundation of China (Grant No. 11971149).

\bibliographystyle{Chicago}
\bibliography{ssnl1_2}
\end{document}